\documentclass[12pt,reqno]{amsart}

\usepackage{epsfig}
\usepackage{amsmath, amsthm, amsfonts, amssymb, mathrsfs}
\usepackage{tikz-cd}
\usepackage{graphicx}

\numberwithin{equation}{section}
\DeclareMathOperator\Hom{Hom}
\DeclareMathOperator\Ext{Ext}

\newcommand{\R}{{\mathbb R}}
\newcommand{\C}{{\mathbb C}}

\newcommand{\Q}{{\mathbb Q}}
\newcommand{\Z}{{\mathbb Z}}

\renewcommand{\Im}{{\operatorname{Im\,}}}

\newtheorem{theo}{{\sc \bf Theorem}}[section]
\newtheorem{cor}[theo]{{\sc \bf Corollary}}
\newtheorem{lem}[theo]{{\sc \bf Lemma}}
\newtheorem{prop}[theo]{{\sc \bf Proposition}}
\newenvironment{example}{\medskip\noindent{\it Example:\/} }{\medskip}

\newenvironment{defin}{\medskip\noindent{\bf Definition:\/} }{\medskip}

\begin{document}

\title{Aspects of Noncommutative Geometry of Bunce-Deddens Algebras}

\author[Klimek]{Slawomir Klimek}
\address{Department of Mathematical Sciences,
Indiana University-Purdue University Indianapolis,
402 N. Blackford St., Indianapolis, IN 46202, U.S.A.}
\email{sklimek@math.iupui.edu}

\author[McBride]{Matt McBride}
\address{Department of Mathematics and Statistics,
Mississippi State University,
175 President's Cir., Mississippi State, MS 39762, U.S.A.}
\email{mmcbride@math.msstate.edu}

\author[Peoples]{J. Wilson Peoples}
\address{Department of Mathematics,
Pennsylvania State University,
107 McAllister Bld., University Park, State College, PA 16802, U.S.A.}
\email{jwp5828@psu.edu}

\thanks{}

\date{\today}

\begin{abstract}
We define and study smooth subalgebras of Bunce-Deddens C$^*$-algebras. We discuss various aspects of noncommutative geometry of Bunce-Deddens algebras including derivations on smooth subalgebras, as well as K-Theory and K-Homology. 
\end{abstract}

\maketitle
\section{Introduction}
The purpose of this paper is to study noncommutative geometry of Bunce-Deddens (BD) algebras \cite{BD1}, \cite{BD2} and their natural smooth subalgebras. Our previous, related investigations of BD algebras are contained in \cite{KMRSW2} and \cite{KM6} and are mostly concerned with classifications of unbounded derivations.

The main objects of study in this paper are smooth subalgebras of Bunce-Deddens algebras. The smooth subalgebras are dense *-subalgebras closed under holomorphic functional calculus and complete in their own stronger locally convex topology. As the name suggests, smooth subalgebras capture differentiable structure in noncommutative geometry, in analogy with ordinary differential geometry and thus are important objects to study. 

The most common way to construct smooth subalgebras in C$^*$-algebras is as smooth elements of an action of a Lie group, see \cite{Bo}. There is a natural circle action on BD algebras, however the resulting smooth subalgebras seem to be too big and in particular don't have a nice classification of derivations on them.

The smooth-subalgebras studied in this paper are more closely related to the work in \cite{N} regarding smooth subalgebras of crossed products $A\rtimes\Z$ of which Bunce-Deddens algebras are examples. A key difference, however, is that the construction in  \cite{N} starts with a dense subalgebra of $A$ with its own Frechet topology, while in our case a choice of such a subalgebra is not obvious. Bunce-Deddens algebras are crossed-product C$^*$-algebras obtained from odometers. Due to the topology of odometers, which are Cantor sets with a minimal action of a homeomorphism, the smooth subalgebras are naturally equipped with inductive limit Frechet (LF) topology.

Derivations naturally arise in differential geometry as vector fields and play an analogous role in noncommutative geometry. We prove that the space of continuous derivations on the smooth subalgebra of a BD algebra modulo inner derivations is just one-dimensional and we exhibit its natural generator.

K-Theory and K-Homology are key invariants in the noncommutative topology of quantum spaces. In addition to reviewing the known $K$-Theory of BD algebras, we include two detailed calculations of $K$-Homology.  One calculation uses Rosenberg and Schochet's Universal Coefficient Theorem, while the other uses Pimsner and Voiculescu's  $6$-term exact sequence. 

Smooth subalgebras of noncommutative spaces are also naturally present in studying spectral triples and in cyclic cohomology which we do not address in this paper. Examples of spectral triples for BD algebras are discussed in \cite{HSWZ}.

The paper is organized as follows.  In section 2 we review supernatural numbers and Bunce-Deddens algebras described as concrete C$^*$-algebras of operators in a Hilbert space.  In section 3 we discuss smooth subalgebras of Bunce-Deddens algebras and their properties.  In section 4 we classify continuous derivations on the smooth subalgebras from section 3.  We also discuss examples that show the necessity of the assumption that the derivations are continuous.  In section 5 we discuss the $K$-Theory and $K$-Homology of  Bunce-Deddens algebras.  Finally in section 6, the appendix, we provide an example based guide for computations in homological algebra involving the Ext functor.  

\section{Preliminaries}
We begin this section by recalling some relevant information regarding supernatural numbers.  
\subsection{Supernatural Numbers}
A {\it supernatural number} $S$ is defined as the formal product:
\begin{equation*} 
S= \prod_{p-\textrm{prime}} p^{\varepsilon_p}, \;\;\; \varepsilon_p \in\{0,1, \cdots, \infty\}\,.
\end{equation*}
When a supernatural number $S$ has the property that $\sum \varepsilon_p < \infty$, we say that $S$ is a finite supernatural number. Otherwise, $S$ is said to be an infinite supernatural number. Given another supernatural number
\begin{equation*}
S'= \prod_{p-\textrm{prime}} p^{\varepsilon_p'},
\end{equation*} 
their product is defined by
\begin{equation*}
SS':= \prod_{p-\textrm{prime}} p^{\varepsilon_p + \varepsilon_p'}\,.
\end{equation*}
We say that a supernatural number $S$ divides $\mathcal{S}$ if there is some supernatural number $S'$ such that $\mathcal{S}=SS'$. In terms of the powers, this is equivalent to  $\varepsilon_p(S)\le\varepsilon_p(\mathcal{S})$ for every prime $p$.

For the remainder of the paper we work with a fixed supernatural number $S$. Consider the set of divisors of $S$:
\begin{equation*}
\mathcal{L}_S=\{l: \; l|S, l<\infty\}.
\end{equation*}
Notice that $(\mathcal{L}_S, \le)$ is a directed set where $l_1 \le l_2$ if and only if $l_1 | l_2 |S$. We say that a sequence $\{l_n\}$ of divisors of $S$ converges to $S$, $\lim_{n\to\infty}l_n=S$, if the exponents converge:
\begin{equation*}
\lim_{n\to\infty}\varepsilon_p(l_n)=\varepsilon_p(S).
\end{equation*}

\subsection{Odometers \cite{D}}
For each pair of divisors $l,k$ satisfying $l \geq k$, consider the ring homomorphism $\pi_{kl}: \Z/ l\Z \rightarrow\Z/ k\Z$ defined by
\begin{equation*}
\pi_{kl}(x) =x\ (\textrm{mod } k)
\end{equation*}
which satisfy 
\begin{equation*}
\pi_{kj} = \pi_{kl} \circ \pi_{lj} \textnormal{ for all } k\le l\le j\,.
\end{equation*}
Then the inverse limit of the system can be denoted as:
\begin{equation*}
\Z/S\Z:=\lim_{\underset{l\in \mathcal{L}_S}{\longleftarrow}} \Z/ l\Z =\left\{\{x_l\}\in\prod\limits_{l\in \mathcal{L}_S}\Z/ l\Z : \pi_{kl}(x_l)=x_k\right\}\,,
\end{equation*}
and let 
\begin{equation*}
\pi_l: \Z/S\Z \ni \{x_l\}\mapsto x_l \in \Z/l\Z
\end{equation*} 
be the corresponding homomorphisms. Notice that for finite supernatural number $S$, the above notation is consistent with the usual meaning of the symbol $\Z/S\Z$. A more interesting observation is that if $S=p^\infty$ for a prime p, then the above limit is equal to $\Z_p$, the ring of $p$-adic integers, see for example \cite{R}.
In general we have the following simple consequence of the Chinese Reminder Theorem.

\begin{prop}
If $S= \prod\limits_{\substack{p-\textrm{prime} \\ {\varepsilon_p \neq 0}}} p^{\varepsilon_p}$, then $\Z/S\Z \cong \prod\limits_{\substack{p-\textrm{prime} \\ {\varepsilon_p \neq 0}}}  \Z/{p^{\varepsilon_p}}\Z$. 
 \end{prop}

If the ring $\Z/S\Z$ is equipped with the Tychonoff topology it forms a compact, Abelian topological group. Thus it has a unique normalized Haar measure $d_Hx$. In addition, if $S$ is an infinite supernatural number then $\Z/S\Z$ is a Cantor set \cite{W}. 

\subsection{Dense Cyclic Subgroup}
$\Z / S \Z$ contains a dense copy of $\Z$, which we now describe. Define a homomorphism $q: \Z \rightarrow \Z/S\Z$ by:
\begin{equation}\label{q_def}
q(x)=\{x\ (\textrm{mod } l)\}\in  \Z/S\Z\subseteq \prod\limits_{l\in \mathcal{L}_S}\Z/ l\Z.
\end{equation}
and let $q_l: \Z \rightarrow \Z/l\Z$ be the usual quotient maps. We have the following simple property:
\begin{equation*}
\pi_l\circ q=q_l\,.
\end{equation*}
From this fact, along with the structure of cylinder sets, the following proposition can be deduced: 
\begin{prop}
The range of $q$ is dense in $\Z/S\Z$. 
\end{prop}


\subsection{Locally Constant Functions \cite{HR}}
We denote by $\mathcal{E}(\Z/S\Z)$ the space of locally constant functions on $\Z/S\Z$. This is a dense subspace of the space of continuous functions on $\Z/S\Z$.  Given an $f\in\mathcal{E}(\Z/S\Z)$, consider the sequence: 
\begin{equation*}
a_f(k)= f(q(k)),\ \ k\in\Z\,.
\end{equation*} 
We have the following observation:
\begin{prop}\label{loc_const}
If $f\in\mathcal{E}(\Z/S\Z)$, then there exists $l\in\mathcal{L}_S$ such that 
\begin{equation*}
a_f(k+l)= a_f(k)
\end{equation*} 
for every $k\in\Z$. Conversely, if $a(k)$ is an $l$-periodic sequence for some $l\in\mathcal{L}_S$, then there is a unique $f\in\mathcal{E}(\Z/S\Z)$ such that $a(k)=a_f(k)$.
\end{prop}
The proof follows from the fact that $f$ is locally constant on $\Z/S\Z$ iff there is $l|S$ such that $f(x)=f(y)$ for all $x=\{x_k\}$ and $y=\{y_k\}$ in $\Z/S\Z$ such that $x_k=y_k$ if $l | k$.

\subsection{Bunce-Dedden Algebras}
To introduce Bunce-Deddens algebras we follow 
\cite{KM6} which more generally described algebras associated with infinite compact monothetic groups.  
Recall that a {\it monothetic} group is a topological group that has a dense cyclic subgroup.  For an infinite supernatural number $S$, the group $\Z/S\Z$ is an infinite compact monothetic group with $\Z$, the image of the map $q$, being its dense cyclic subgroup.  

Let $H$ be the $\ell^2$ Hilbert space of the cyclic subgroup, that is 
\begin{equation*}
H=\ell^2(\{q(l)\}_{l\in\Z})\,,
\end{equation*} 
which is naturally isomorphic with $\ell^2(\Z)$.  Let $\{E_l\}_{l\in\Z}$ be the canonical basis in $H$ and let $U:H\to H$ be the shift operator on $H$:
\begin{equation*}
UE_l = E_{l+1}\,.
\end{equation*}

For a continuous function $f\in C(\Z/S\Z)$ we define an operator $M_f:H\to H$ via:
\begin{equation*}
M_fE_l = f(q(l))E_l\,.
\end{equation*}
Notice that $M_f$ is a diagonal multiplication operator on $H$.  Due to the density of the subgroup $\{q(l)\}_{l\in\Z}$, we immediately obtain:

\begin{equation*}
\|M_f\|= \underset{l\in\Z}{\textrm{sup }}|f(q(l))|=\underset{x\in \Z/S\Z}{\textrm{sup }}|f(x)| = \|f\|_\infty\,.
\end{equation*}
The algebra of operators generated by the $M_f$'s is thus isomorphic to $C(\Z/S\Z)$. So, while the $M_f$'s carry all the information about the space $\Z/S\Z$, the operator $U$ reflects the dynamics of the map $\beta$ on $\Z/S\Z$ given by addition of the generator of the cyclic subgroup:
\begin{equation}\label{beta_def}
\beta(x)=x+q(1).
\end{equation}
This relationship is summarized by:
\begin{equation*}
UM_fU^{-1}=M_{f\circ\beta}\,.
\end{equation*}

We define the algebra $B_S$ to be the C$^*$-algebra generated by operators $U$ and $M_f$:
\begin{equation*}
B_S=C^*\{U,M_f: f\in C(\Z/S\Z)\}\,.
\end{equation*} 
$B_S$ is isomorphic with the crossed product algebra:
\begin{equation*}
B_S\cong C(\Z/S\Z)\rtimes_\beta \Z\,.
\end{equation*} 
Indeed, observe that $\Z$ is amenable, the action of $\Z$ on $\Z/S\Z$ given by $\beta$ is a free action, and $\beta$ is a minimal homeomorphism. Thus, the crossed product is simple and equal to the reduced crossed product, see, for instance, \cite{F}. Clearly, the operators $U$ and $M_f$ define a representation of the crossed product $C(\Z/S\Z)\rtimes_\beta\Z$, and the algebra they generate must be isomorphic to it, by simplicity of the crossed product.  This crossed product is known as the Bunce-Deddens algebra.

\subsection{Polynomial Subalgebra}
Let $\chi$ be a character of the group $\Z/S\Z$. We have $\chi\in\mathcal{E}(\Z/S\Z)$ and so $\chi$ comes from a $l$ periodic function on $\Z$, by an abuse of notation also denoted by $\chi$. Any $l$ periodic character is of the form $\chi_l^k$ where
$$\chi_l(q(x)) = e^{2\pi ix/l},$$
and $k$ is an integer.

The algebra $B_S$ has a natural dense $*$-subalgebra $\mathcal{B}_S$ of polynomials in $U$, $U^{-1}$, and the $M_\chi$'s, where $\chi$ is a character of $\Z/S\Z$. We have:
\begin{equation*}
\mathcal{B}_S = \left\{\sum_nU^nM_{f_n}: f_n\in\mathcal{E}(\Z/S\Z),\textrm{ finite sums}\right\}\,.
\end{equation*}
Alternatively, any element $a$ in $\mathcal{B}_S$ can be written as
\begin{equation}\label{CharDecomp}
a=\sum_{k=0}^{l-1}F_k(U)M_{\chi_l^k}\,,
\end{equation}
where $F_k$ are trigonometric polynomials.

\subsection{Label Operator}
It is often convenient to use the following diagonal label operator on $H$:  
\begin{equation*}
\mathbb{L}E_l = lE_l\,.
\end{equation*}
If $\{b(l)\}_{l\in\Z}$ is a bounded sequence, then by the functional calculus, $b(\mathbb L)$ is a bounded diagonal operator given by:
\begin{equation*}
b(\mathbb L) E_l= b(l) E_l\,,
\end{equation*}
and all bounded diagonal operators can be expressed this way.  Moreover, we have the following commutation relation:
\begin{equation}\label{comm_rel}
b(\mathbb L)U = Ub(\mathbb L + I)\,.
\end{equation}

\section{Smooth Subalgebras}
\subsection{Definitions}
To define the space of smooth elements of the Bunce-Deddens algebra we need the following terminology.  Let $\{a_n\}$ be a sequence of complex numbers.  We say $\{a_n\}$ is {\it Rapid Decay}, RD, if for every $k\ge0$, there exists a constant $C_k$ such that
\begin{equation*}
n^k|a_n|\le C_k\,.
\end{equation*}
We say a family of locally constant functions is {\it Uniformly Locally Constant}, ULC, if there exists a divisor $l$ of $S$ such that for every $f$ in the family we have
\begin{equation*}
f(x+l) = f(x)
\end{equation*}
for all $x\in\Z/S\Z$.   

We define the space of smooth elements of the Bunce-Deddens algebra, $B_S^\infty$, to be the space of elements in $B_S$ whose Fourier coefficients are ULC and norms are RD.  Using Equation \ref{comm_rel} the ULC condition on Fourier coefficients can be written as:
\begin{equation*}
B_S^\infty=\left\{b= \sum_{n\in\Z}U^nM_{f_n} : \{\|f_n\|\}\textrm{ is }RD, \textrm{ there is an }l|S,\,\,U^lbU^{-l}=b\right\}\,.
\end{equation*}
It's immediate that $B_S^\infty$ is indeed a nonempty subset of $B_S$. 

\subsection{M-Norms}
Our next goal is to prove that $B_S^\infty$ is a $*$-subalgebra of $B_S$. This is achieved in several steps. The first is the definition of a derivation \cite{KR}.

Let $A$ be an algebra. A linear map $\delta:A\to A$ is called a {\it derivation} if the Leibniz rule holds:
\begin{equation*}
\delta(ab) = a\delta(b) + \delta(a)b
\end{equation*}
for all $a,b\in A$.   An example of a derivation $B^\infty_S \to B^\infty_S$ is given by
\begin{equation*}
\delta_{\mathbb L}(b) = [\mathbb L,b].
\end{equation*}
Since $[\mathbb L,U^n]=nU^n$ and $[\mathbb L,M_f]=0$, we see that
\begin{equation*}
\textrm{if}\quad b = \sum_{n\in\Z}U^nM_{f_n}\,,\quad\textrm{then}\quad\delta_{\mathbb L}(b) = \sum_{n\in\Z}U^n nM_{f_n}\,.
\end{equation*}
Thus if $\{\|f_n\|\}$ are RD, then so too are $\{n\|f_n\|\}$ and hence $\delta_{\mathbb L}(b)$ is in $B_S^\infty$.  This derivation is very fundamental, and is used extensively in the sequel. 

In the following lemma we describe a class of norms on $B_S$ that we call $M$-norms. These $M$-norms are used to provide an alternative description of the RD condition in the definition of $B_S^\infty$ (see Proposition \ref{B_infty_descript} below).
\begin{lem}\label{M-norms}
Let $A$ be an algebra with a submultiplicative norm $\|\cdot\|$ and let $\delta:A\to A$ be any derivation.  For any $a\in A$ and $M$ a nonnegative integer, define recursively:
\begin{equation*}
\|a\|_{M+1} = \|a\|_M + \|\delta(a)\|_M\,,
\end{equation*}
with $\|a\|_0:=\|a\|$.  Then, $\|\cdot\|_M$ is a submultiplicative norm on $A$.  Moreover
\begin{equation*}
\|a\|_M = \sum_{j=0}^M\
\begin{pmatrix}
M \\ j
\end{pmatrix}\|\delta^{(j)}(a)\|\,.
\end{equation*}
\end{lem}

\begin{proof}
Proving that $\|\cdot\|_M$ is a norm is on $A$ is an immediate consequence of induction and the fact that $\|\cdot\|$ is a norm on $A$.  To prove submultiplicativeness, induction can also be used.  The base case $M=0$ is immediate as $\|\cdot\|$ is submultiplicative.  Thus given an $M$, suppose $\|\cdot\|_M$ is submultiplicative. We have the following calculation:
\begin{equation*}
\begin{aligned}
\|ab\|_{M+1} &= \|ab\|_M + \|\delta(ab)\|_M = \|ab\|_M + \|\delta(a)b+ a\delta(b)\|_M\\
&\le \|a\|_M\|b\|_M + \|\delta(a)\|_M\|b\|_M + \|a\|_M\|\delta(b)\|_M \\
&\le(\|a\|_M + \|\delta(a)\|_M)(\|b\|_M + \|\delta(b)\|_M) = \|a\|_{M+1}\|b\|_{M+1}\,.
\end{aligned} 
\end{equation*}
Thus $\|\cdot\|_M$ is submultiplicative.  

It remains to prove the formula in the last part of the statement of the lemma. Again, we use induction. The base case $M=0$ is trivial. Suppose the formula is true for a given $M$ and consider
\begin{equation*}
\|a\|_{M+1} = \|a\|_M + \|\delta(a)\|_M = \sum_{j=0}^M
\begin{pmatrix}
M \\ j
\end{pmatrix}\|\delta^{(j)}(a)\| + \sum_{j=0}^M
\begin{pmatrix}
M \\ j
\end{pmatrix}\|\delta^{(j+1)}(a)\|\,.
\end{equation*}
Resumming the second sum and using the binomial coefficient reduction formula, we obtain
\begin{equation*}
\begin{aligned}
\|a\|_{M+1} &= \|a\|_0 + \sum_{j=1}^M\left[
\begin{pmatrix}
M \\ j
\end{pmatrix} + 
\begin{pmatrix}
M \\ j-1
\end{pmatrix}\right]\|\delta^{(j)}(a)\| + \|a\|_{M+1} \\
&=\sum_{j=0}^{M+1}
\begin{pmatrix}
M+1 \\ j
\end{pmatrix}\|\delta^{(j)}(a)\|\,. 
\end{aligned}
\end{equation*}
This completes the proof.
\end{proof}

For the remainder of the paper we focus on the following particular $M$-norm on $B_S$:
\begin{equation*}
\|b\|_M = \sum_{j=0}^M\
\begin{pmatrix}
M \\ j
\end{pmatrix}\|\delta_{\mathbb L}^{(j)}(b)\|\,.
\end{equation*}
In general, this norm is not well-defined on all of $B_S$ as $\mathbb L$ is an unbounded operator. We must first properly define $\|\delta_{\mathbb L}^{(j)}(b)\|$. We do this in the following discussion. Let $\mathcal{D}\subseteq\ell^2(\Z)$ be the following subspace:
\begin{equation*}
\mathcal{D} = \left\{x\in\ell^2(\Z): x = \sum_{k\in\Z}x_kE_k\textrm{ with }\{x_k\}\textrm{-RD}\right\}\,.
\end{equation*}
Then pick $x\in\mathcal{D}$ and a $b\in B_S$, with $\{b_{kl}\}$ the matrix coefficients of $b$ with respect to the canonical basis $\{E_l\}$, and consider the following calculation:
\begin{equation*}
\begin{aligned}
\|\delta_{\mathbb L}(b)x\|^2 &= \left\|\sum_{k\in\Z}x_k\delta_{\mathbb L}(b)E_k\right\|^2 = \left\|\sum_{k,l\in\Z}lx_kb_{kl}E_l-\sum_{k,l\in\Z}kx_kb_{kl}E_l\right\|^2 \\
&= \left\|\sum_{l\in\Z}\left(\sum_{k\in\Z}(l-k)b_{kl}x_k\right)E_l\right\|^2\,.
\end{aligned}
\end{equation*}
The inner sum over $k$ is finite since $\{b_{kl}\}$ is a bounded sequence as $b\in B_S$, $\{x_k\}$ is a RD sequence and $l-k$ is a polynomial in $k$. Therefore we can define
\begin{equation*}
\|\delta_{\mathbb L}(b)x\|^2 = \sum_{l\in\Z}\left|\sum_{k\in\Z}(l-k)b_{kl}x_k\right|^2
\end{equation*}
which may or may not be finite, but is a series of nonnegative terms over $l$.  Thus for  $b\in B_S$ we define $\|\delta_{\mathbb L}(b)\|$ by
\begin{equation*}
\|\delta_{\mathbb L}(b)\| = \underset{x\in\mathcal{D}}{\textrm{sup }}\frac{\|\delta_{\mathbb L}(b)x\|}{\|x\|}\,.
\end{equation*}
Using similar considerations, we can define $\|\delta_{\mathbb L}^{(j)}(b)\|$. Now $\|b\|_M\in[0,\infty]$ makes sense for every $b\in B_S$.

\subsection{Basic Properties}
Next we want to relate the RD condition for $b\in B_S^\infty$ to finiteness of $M$-norms.  This is done using a $1$-parameter group of automorphisms of $B_S$ that is given by the following equation:
\begin{equation*}
\rho_\theta(b) = e^{2\pi i\theta\mathbb{L}}be^{-2\pi i\theta\mathbb{L}}\quad\textrm{ for }b\in B_S,
\end{equation*}
where $\theta\in\R/2\pi\Z$. We have the following formulas:
\begin{equation*}
\rho_\theta(U)=e^{2\pi i\theta}U\textrm{ and } \rho_\theta(b(\mathbb{L}))=b(\mathbb{L})\,.
\end{equation*}
It immediately follows that $\rho_\theta:\mathcal{B}_S\to\mathcal{B}_S$ and that $\rho_\theta:B_S^\infty\to B_S^\infty$.

Define $E:B_S\to C^*\{M_f:f\in C(\Z/S\Z)\}\cong C(\Z/S\Z)$ via
\begin{equation*}
E(b) =\int_0^1\rho_\theta(b)\,d\theta\,.
\end{equation*}
It's easily checked that $E$ is an expectation on $B_S$.  For a $b\in B_S$ we define the {\it $n$-th Fourier coefficient}, $b_n$ by the following:
\begin{equation*}
b_n=E(U^{-n}b)= \int_0^1\rho_\theta(U^{-n}b)\,d\theta\, = \int_0^1e^{-2\pi in\theta}U^{-n}\rho_\theta(b)\,d\theta.
\end{equation*}
From this definition, it's clear that $b_n\in C^*\{M_f:f\in C(\Z/S\Z)\}$.  We have the following lemma:
\begin{lem}\label{Fourier_coeff}
Let $b$ and $b'$ be elements of $B_S$.  If $b_n=b_n'$ for every $n$, then $b=b'$.  Moreover, if $\{\|b_n\|\}$ is a RD sequence, then
\begin{equation*}
b=\sum_{n\in\Z}U^nb_n
\end{equation*}
where the sum is norm convergent.
\end{lem}
\begin{proof}
The first part follows from the standard techniques from Fourier analysis, see for example \cite{K}.  

If $\{\|b_n\|\}$ is a RD sequence, then, there exists a constant $C$ such that $(|n|^2+1)\|b_n\|\le C$.  Thus
\begin{equation*}
\left\|\sum_{|n|\le j}U^nb_n\right\|\le \sum_{|n|\le j}\|b_n\|\le \sum_{|n|\le j}\frac{1}{|n|^2+1}\le \sum_{n\in\Z}\frac{1}{|n|^2+1}
\end{equation*}
which is an absolutely convergent sum.  Thus by the Weierstrass Test, $\sum_n U^nb_n$ is a norm convergent sum.  It follows that $b_n$ are the Fourier coefficients of $b\in B_S$ and $b=\sum_{n\in\Z}U^nb_n\,.$
\end{proof}

With this lemma, we have the following equivalent definition of $B_S^\infty$:
\begin{prop}
\label{B_infty_descript}
\begin{equation*}
B_S^\infty = \{b\in B_S: \|b\|_M<\infty\,,\textrm{ for every }M,\textrm{ there is an }l|S,\,\,U^lbU^{-l}=b\}\,.
\end{equation*}
\end{prop}
\begin{proof}
In order to prove the proposition, we only need to check that finite $M$-norms is equivalent to RD Fourier coefficients.  If $b\in B_S^\infty$, then $\{\|f_n\|\}$ are RD.  A straightforward calculation shows that:
\begin{equation*}
\textrm{if }b = \sum_{n\in\Z} U^nM_{f_n}\,,\quad\textrm{then}\quad\delta_{\mathbb L}^{(j)}(b) = \sum_{n\in\Z}n^jU^nM_{f_n}\,,
\end{equation*}
and consequently we have:
\begin{equation*}
||\delta_{\mathbb L}^{(j)}(b)||\leq \sum_{n\in\Z}|n|^j||f_n||\,.
\end{equation*}
It follows that $\|b\|_M$ is finite for every $M$.  

On the other hand, let $b\in B_S$ and assume that $\|b\|_M$ is finite for every $M$.  Then we get the following formula for the Fourier coefficients $M_{f_n}$ of $b$:
\begin{equation*}
E\left(U^{-n}\delta_{\mathbb L}^{(j)}(b)\right) = n^jE(U^{-n}b) = n^jM_{f_n}\,,
\end{equation*} and hence
\begin{equation*}
n^j\|f_n\| = n^j\|M_{f_n}\|\le \|\delta_{\mathbb L}^{(j)}(b)\|\le \|b\|_M\,.
\end{equation*}
Therefore if $\|b\|_M$ is finite for every $M$, it follows that $\{\|f_n\|\}$ are RD.  
\end{proof}

We now give a brief discussion of the natural topology on $B_S^\infty$ and continuity of linear maps on it. Define the following spaces:
\begin{equation*}
B_{S,\,l}^\infty = \{b\in B_S^\infty : U^lbU^{-l}=b\textrm{ for some }l|S\}\,.
\end{equation*}
Notice that $B_{S,\,l}^\infty$ is a closed subspace of $B_S^\infty$.  Moreover if $l$ and $l'$ are two different divisors of $S$ with $l<l'$, then it follows that $B_{S,\,l}^\infty\subsetneqq B_{S,\,l'}^\infty$.  Also $B_{S,\,l}^\infty$ is a Fr\'echet space with respect to $\|\cdot\|_M$ and
\begin{equation*}
B_S^\infty = \bigcup_{l|S} B_{S,\,l}^\infty
\end{equation*}
as a strict inductive limit.  To establish some of the results below it is also useful to represent this inductive limit in the following way. Let $1<l_1|\,l_2|\cdots$ be a strictly increasing sequence of divisors of $S$. Notice that 
\begin{equation*}
B_S^\infty = \bigcup_{n=1}^\infty B_{S,\,l_n}^\infty\,.
\end{equation*}

Let $T:B_S^\infty\to B_S^\infty$ be a linear map.  Then by the general property of inductive limits, see \cite{RS}, $T$ is continuous if and only if $T:B_{S,\,l_n}^\infty\to B_S^\infty$ is continuous for every $n$. Additionally we have the following useful result.

\begin{lem}\label{reed_simon_1}
Let $T:B_S^\infty\to B_S^\infty$ be a continuous linear map.  Then for every $n$, there exists an $m$ such that $T(B_{S,\,l_n}^\infty)\subseteq B_{S,\,l_m}^\infty$.
\end{lem}

\begin{proof}
Since $B_{S,\,l_n}^\infty$ is separable, choose a sequence $\{g_k\}_{k=1}^\infty$ with $g_k\in B_{S,\,l_n}^\infty$ for every $k$, such that $\textrm{span}_k\{g_k\}$ is dense in $B_{S,\,l_n}^\infty$ and $\|g_k\|_1=1$.  Define $f_k = g_k/k$ and so $f_k\to0$ in $B_{S,\,l_n}^\infty$.  Since $T$ is continuous, we have $T(f_k)\to0$ in $B_S^\infty$.  Thus by Theorem 5.17 in \cite{RS}, it follows that there exists an $m$ such that $T(f_k)\subseteq B_{S,\,l_m}^\infty$ for every $k$.  However,
\begin{equation*}
T(\textrm{span}_k\{g_k\})=T(\textrm{span}_k\{f_k\})\subseteq B_{S,\,l_m}^\infty\,.
\end{equation*}
But $\textrm{span}_k\{g_k\}$ is dense in $B_{S,\,l_n}^\infty$ and $B_{S,\,l_m}^\infty$ is complete.  Thus it follows that $T(B_{S,\,l_n}^\infty)\subseteq B_{S,\,l_m}^\infty$.
\end{proof}

The following lemma is a special case, adapted to our scenario, of Theorem 5.2 in \cite{RS}.
\begin{lem}\label{reed_simon_2}
Let $T:B_{S,\,l_n}^\infty\to B_{S,\,l_m}^\infty$ be a linear map.  Then $T$ is continuous if and only if for every $M$, there exist an $M'$ and a constant $C=C(M)$ such that for every $b\in B_{S,\,l_n}^\infty$, the following inequality holds:
\begin{equation*}
\|T(b)\|_M \le C\|b\|_{M'}\,.
\end{equation*}
\end{lem}

Together, Lemmas \ref{reed_simon_1} and \ref{reed_simon_2} imply the following theorem:
\begin{theo}
$T:B_S^\infty\to B_S^\infty$ be a linear map.  $T$ is continuous if and only if for every $n$, there exists an $m$ such that $T(B_{S,\,l_n}^\infty)\subseteq B_{S,\,l_m}^\infty$ and for every $M$, there exist an $M'$ and a constant $C=C(M)$ such that for every $b\in B_S^\infty$, the following inequality holds:
\begin{equation*}
\|T(b)\|_M \le C\|b\|_{M'}\,.
\end{equation*}
\end{theo}

\begin{cor} $\delta_{\mathbb L}:B_S^\infty\to B_S^\infty$ is a continuous derivation. 
\end{cor}

To conclude this section, we prove that $B^\infty_S$ is a $*$-subalgebra of $B_S$ closed under the holomorphic functional calculus \cite{Bo}. 

\begin{prop}
$B_S^\infty$ is a $*$-subalgebra of $B_S$ with respect to $\|\cdot\|_M$.
\end{prop}

\begin{proof}
Since $\|\cdot\|_M$ is a norm and the relation $a=U^lau^{-l}$ for some divisor $l$ of $S$ is linear, it follows that $B_S^\infty$ is closed under addition and scalar multiplication.  Moreover since $\|\cdot\|_M$ is $*$-preserving, it follows that $B_S^\infty$ is closed under taking $*$'s.  It only remains to show that $B_S^\infty$ is closed under multiplication.  Let $a$ and $b$ be two elements in $B_S^\infty$ and let $l_1$ and $l_2$ be two divisors of $S$ such that
\begin{equation*}
a= U^{l_1}aU^{-l_1} \quad\textrm{and}\quad b = U^{l_2}bU^{-l_2}\,.
\end{equation*}
Set $l=\textrm{lcm}(l_1,l_2)$, then
\begin{equation*}
U^labU^{-l} = U^laU^{-l}U^lbU^{-l} = ab\,.
\end{equation*}
Finally, since $\|\cdot\|_M$ is submultiplicative by Lemma \ref{M-norms}, it follows that $ab\in B_S^\infty$.
\end{proof}

\subsection{Holomorphic Calculus} 
We verify below that $B_S^\infty$ is closed under the holomorphic functional calculus and hence has the same K-Theory as $B_S$.
\begin{lem}\label{b_inverse}
Let $b\in B_S^\infty$ such that $b$ is invertible in $B_S$.  Then, $b^{-1}\in B_S^\infty$.
\end{lem}

\begin{proof}
Since $b\in B_S^\infty$, there exists a divisor $l$ of $S$ such that $U^lbU^{-l}= b$.  Since $b$ and $U$ are invertible in $B_S$, it immediately follows that $U^lb^{-1}U^{-l}=b^{-1}$.  

To check that the $\|b^{-1}\|_M$ is finite, notice that
\begin{equation*}
\begin{aligned}
\|b^{-1}\|_1 &= \|b^{-1}\| + \|\delta_{\mathbb{L}}(b^{-1})\|\le \|b^{-1}\| + \|b^{-1}\|^2\|\delta_{\mathbb{L}}(b)\| \\
&\le \|b^{-1}\|^2\|b\| + \|b^{-1}\|^2\|\delta_{\mathbb{L}}(b)\| = \|b^{-1}\|^2\|b\|_1\,.
\end{aligned}
\end{equation*}
By definition we know that 
\begin{equation*}
\|b\|_{M+1} = \|b\|_M + \|\delta_{\mathbb{L}}(b)\|_M\,,
\end{equation*}
and so by a similar argument we obtain
\begin{equation*}
\|b^{-1}\|_2 \leq \|b^{-1}\|_1^2\|b\|_2\,.
\end{equation*}
Proceeding inductively we get
\begin{equation*}
\|b^{-1}\|_{M+1}\le \|b^{-1}\|_M^2\|b\|_{M+1}\,.
\end{equation*}
Hence $\|b^{-1}\|_M$ is finite for all $M$ and thus $b^{-1}\in B_S^\infty$.
\end{proof}

\begin{prop}
$B_S^\infty$ is closed under the holomorphic functional calculus.  That is, given $b\in B_S^\infty$ and a function $f$ that is holomorphic on an open domain containing the spectrum of $b$, we have $f(b)\in B_S^\infty$.
\end{prop}

\begin{proof}
Let $b\in B_S^\infty$ and let $f$ be a holomorphic function on some open set containing $\sigma(b)$ and $C$ a contour aroud the spectrum.    Write $f$ in Cauchy integral form:
\begin{equation*}
f(z) = \frac{1}{2\pi i}\int_C \frac{f(\zeta)}{\zeta - z}\,d\zeta
\end{equation*}
which defines $f(b)$ as
\begin{equation*}
f(b) = \frac{1}{2\pi i}\int_C f(\zeta)(\zeta- b)^{-1}\,d\zeta\,.
\end{equation*}
Since $b\in B_S^\infty$, there exists a divisor $l$ of $S$ such that $b=U^lbU^{-l}$.  It follows from the Cauchy integral form above for that $f(b) = U^lf(b)U^{-l}$.  

From Lemma \ref{b_inverse}, it follows that if $b\in B_{S,\,l}^\infty$ for some divisor $l$ of $S$ and $b$ is invertible in $B_S$, then $b^{-1}\in B_{S,\,l}^\infty$. Consequently,  $(\zeta- b)^{-1}\in B_{S,\,l}^\infty$ for $\zeta\notin \sigma(b)$. It follows from completeness of $B_{S,\,l}^\infty$ that $f(b)\in B_{S,\,l}^\infty\subseteq B_S^\infty$.  This completes the proof.
\end{proof}

\section{Classification of Derivations}
The main goal in this section is to classify continuous derivations $\delta: B_S^\infty\to B_S^\infty$. The results below are based on ideas from \cite{KMRSW2}, which described a classification of unbounded derivations $\delta:\mathcal{B}_S\to B_S$. Here $\mathcal{B}_S$ is the polynomial dense subalgebra of $B_S$. The reference \cite{KM6} contains a generalization of this classification to $C^*$-algebras associated to monothetic groups.  

\subsection{Preliminaries}
We begin with recalling the basic concepts.  Let $A$ be a Banach algebra and $\mathcal{A}$ be a dense subalgebra of $A$. We say a derivation $\delta:\mathcal{A}\to A$ is {\it inner} if there is a $x\in A$ such that
\begin{equation*}
\delta(a) = [x,a]\,
\end{equation*}
for $a\in\mathcal{A}$.  We say a derivation $\delta:\mathcal{A}\to A$ is {\it approximately inner} if there are $x_n\in A$  such that 
\begin{equation*}
\delta(a) = \lim_{n\to\infty}[x_n,a]
\end{equation*} 
for $a\in\mathcal{A}$.  

The main result of \cite{KMRSW2} is:
\begin{theo}
Suppose $\delta:\mathcal{B}_S\to B_S$ is a derivation in $B_S$ with $S$ infinite. Then there exists a unique constant $C$ such that 
\begin{equation*}
\delta= C\delta_{\mathbb L} + \tilde\delta
\end{equation*}
where $\tilde\delta$ is approximately inner.
\end{theo}


Compared to \cite{KMRSW2}, here we are studying derivations on a larger domain but with smaller range and we need an added assumption of continuity of $\delta$.  It turns out that for continuous derivations $\delta: B_S^\infty\to B_S^\infty$ there is a similar decomposition as above; the key difference is that $\tilde{\delta}$ is inner, not just approximately inner.

Given $n\in\Z$, a continuous derivations $\delta: B_S^\infty\to B_S^\infty$ is said to be a {\it $n$-covariant derivation} if the relation 
\begin{equation*}
\rho_\theta^{-1}\delta\rho_\theta(a)= e^{-2\pi in\theta} \delta(a)
\end{equation*}
holds.  When $n=0$ we say the derivation is invariant.  With this definition, we point out that $\delta_{\mathbb{L}}:B_S^\infty\to B_S^\infty$ is an invariant continuous derivation.

\begin{defin}\label{Fou_comp1}
If $\delta$ is a continuous derivation on $B_S$, the {\it $n$-th Fourier component} of $\delta$ is defined as: 
\begin{equation*}
\delta_n(b)= \int_0^1 e^{2\pi in\theta} \rho_\theta^{-1}\delta\rho_\theta(b)\, d\theta\,.
\end{equation*}
\end{defin}

Notice that the integral in the above definition makes sense because of the continuity of $\delta$ and the continuity of the map $\theta\to\rho_\theta$. 

\subsection{Covariant Derivations}
The key step in classifying derivations is to analyze the properties of their Fourier components. This will be done in steps leading to the formulation of the main result. We will need below the standard observation that if $\delta_1$ and $\delta_2$ are continuous derivations then their commutator $[\delta_1,\delta_2]$ is also a continuous derivation.

\begin{lem}\label{est_delta_n}
If $\delta:B_S^\infty\to B_S^\infty$ is a continuous derivation then its \it $n$-th Fourier component $\delta_n$ is also a continuous derivation on $B_S^\infty$.  Moreover, for every $M\ge0$ and every $k\ge0$, there exists $M'\ge0$ and a constant $C_k=C_k(M)$ such that, for every $b\in B_S^\infty$:
\begin{equation*}
n^k\|\delta_n(b)\|_M \le C_k\|b\|_{M'}\,.
\end{equation*}
\end{lem}

\begin{proof}
Since $\delta$ is a continuous derivation on $B_S^\infty$, there exists a constant $C=C(M)$ and an $M'\ge0$ such that
\begin{equation*}
\|\delta(b)\|_M \le C\|b\|_{M'}\,.
\end{equation*}
Using this and the fact that M-norms are $\rho_\theta$ invariant, we have
\begin{equation*}
\|\delta_n(b)\|_M = \left\|\int_0^1e^{2\pi in\theta}\rho_\theta^{-1}\delta\rho_\theta(b)\,d\theta\right\|_M\le \int_0^1\|\rho_\theta^{-1}\delta\rho_\theta(b)\|_M\,d\theta \le C\|b\|_{M'}\,.
\end{equation*}
Using integration by parts, the continuity of $[\delta_{\mathbb{L}},\delta]$, and the continuous differentiability of $\theta\mapsto \rho_\theta^{-1}\delta\rho_\theta$ on $B_S^\infty$, we have the following calculation:
\begin{equation*}
\begin{aligned}
2\pi in\delta_n(b) &= \int_0^1 2\pi ine^{2\pi in\theta}\rho_\theta^{-1}\delta\rho_\theta(b)\,d\theta = \int_0^1\frac{d}{d\theta}\left(e^{2\pi in\theta}\right)\rho_\theta^{-1}\delta\rho_\theta(b)\,d\theta \\
&=-\int_0^1e^{2\pi in\theta}\frac{d}{d\theta}\left(\rho_\theta^{-1}\delta\rho_\theta(b)\right)\,d\theta=-2\pi i\int_0^1e^{2\pi in\theta}\rho_\theta^{-1}[\delta_{\mathbb{L}},\delta]\rho_\theta(b)\,d\theta \\
&=-2\pi i\left([\delta_{\mathbb{L}},\delta]\right)_n(b)\,,
\end{aligned}
\end{equation*}
where $\left([\delta_{\mathbb{L}},\delta]\right)_n$ is the $n$-th Fourier component of $[\delta_{\mathbb{L}},\delta]$. Using induction on the order of derivatives in terms of $\theta$, the result follows.
\end{proof}

In general, the Fourier series for a derivation constructed as above does not converge to the derivation in norm. We only have the usual Ces\`aro mean convergence result for Fourier components of $\delta$: if $\delta$ is a derivation on $B_S^\infty$ then
\begin{equation}\label{Ces_eq}
\delta(b)=\lim_{M\rightarrow\infty} \frac{1}{M+1} \sum_{j=0}^M \left(\sum_{n=-j}^j \delta_n(b)\right)\,,
\end{equation}
for every $b\in B_S^\infty$, see Lemma 4.2 in \cite{KMRSW2} for more details. In particular, the Fourier components $\delta_n$ completely determine the derivation $\delta$. However, for continuous derivations on $B_S^\infty$, the above lemma implies that Fourier components $\delta_n$ are RD. Hence we have the following statement.

\begin{prop}\label{Fourier_series_norm_conv}
Let $\delta: B_S^\infty\to B_S^\infty$ be a continuous derivation, then
\begin{equation*}
\delta(b) = \sum_{n\in\Z} \delta_n(b)
\end{equation*}
for every $b\in B_S^\infty$.  Here the sum is norm convergent.
\end{prop}
\begin{proof}
We will use the following version of the estimate from Lemma \ref{est_delta_n}: for every $M\ge0$ and every $k\ge0$, there exists $M'\ge0$ and a constant $C_k=C_k(M)$ such that, for every $b\in B_S^\infty$ we have
\begin{equation*}
\left(|n|^k+1\right)\|\delta_n(b)\|_M\le C_k\|b\|_{M'}\,.
\end{equation*}
Thus, considering the tail end of the infinite series and using the above inequality with $k=2$, we have
\begin{equation*}
\left\|\sum_{|j|\ge n}\delta_j(b)\right\|_M\le \sum_{|j|\ge n}\|\delta_j(b)\|_M\le C_2\|b\|_{M'}\sum_{|j|\ge n}\frac{1}{|j|^2+1}\,,
\end{equation*}
which clearly goes to zero as $n$ goes to infinity.  Thus the series is norm convergent.  Since the series is norm convergent, the formula $ \sum_{n\in\Z}\delta_n(b)$ defines a continuous derivation on $B_S^\infty$ with the same Fourier coefficients as $\delta$ hence it must be equal to $\delta$.
\end{proof}  

\begin{prop}\label{delta_n_cov}
Let $\delta:B_S^\infty\to B_S^\infty$ be a continuous derivation.  Then $\delta_n:B_S^\infty\to B_S^\infty$ is a continuous $n$-covariant derivation.
\end{prop}
\begin{proof}
We previously noticed that $\delta_n$ is a derivation and is well-defined on $B_S^\infty$.  Since $\delta$ is a continuous derivation and the automorphism $\rho_\theta$ is continuous, it follows that $\delta_n$ is also continuous. The following computation verifies that $\delta_n$ is $n$-covariant:
\begin{equation*}
\rho_\theta^{-1}\delta_n\rho_\theta(b) = \int_0^1 e^{2\pi in\varphi} \rho_\theta^{-1}\rho_\varphi^{-1}\delta\rho_\varphi\rho_\theta(b)\, d\varphi = \int_0^1 e^{2\pi in\varphi}\rho_{\theta + \varphi}^{-1}\delta\rho_{\theta + \varphi}(b)\, d\varphi\,.
\end{equation*}
Changing to new variable $\theta + \varphi$, and using the translation invariance of the measure, it now follows that $\rho_\theta^{-1}\delta_n\rho_\theta(b)= e^{-2\pi in\theta} \delta_n(b)$. 

Since $\delta:B_S^\infty\to B_S^\infty$, for any $b\in B_S^\infty$ there exists a divisor $l$ of $S$ such that
\begin{equation*}
\delta(b) = U^l\delta(b)U^{-l}. 
\end{equation*}
 It follows from the definition of $\delta_n$ that
\begin{equation*}
\delta_n(b) = U^l\delta_n(b)U^{-l}\,.
\end{equation*}
Hence $\delta_n(b)\in B_S^\infty$.
\end{proof}

It turns out that $n$-covariant derivations can be described explicitly. This was done in \cite{KMRSW2} for unbounded $n$-covariant derivations $\delta:\mathcal{B}_S\to B_S$.  We state those results since they are useful below.

\begin{prop}\label{covariant_der_B_S} \cite{KMRSW2}
Let $\delta:\mathcal{B}_S\to B_S$ be an $n$-covariant derivation where $n\neq0$.  There exists an $F\in C(\Z/S\Z)$ such that
\begin{equation*}
\delta=[U^nM_F,\cdot]\,,
\end{equation*}
so $\delta$ is an inner derivation.
\end{prop} 

\begin{prop}\label{invariant_der_B_S} \cite{KMRSW2}
Let $\delta:\mathcal{B}_S\to B_S$ be a $0$-covariant (invariant) derivation. There exists a unique constant $C$ such that 
\begin{equation*}
\delta= C\delta_{\mathbb L} + \tilde{\delta}\,,
\end{equation*} 
where $\tilde{\delta}$ is an approximately inner derivation.
\end{prop}

Using Proposition \ref{covariant_der_B_S}, for $n\neq0$, we classify all continuous $n$-covariant derivations on $B_S^\infty$.  

\begin{prop}\label{cont_covariant_der_B_S_infty}
Let $\delta:B_S^\infty\to B_S^\infty$ be a continuous $n$-covariant derivation on $B_S^\infty$. There exists a $F\in \mathcal{E}(\Z/S\Z)$ such that
\begin{equation*}
\delta(b) = \left[U^nM_{F},b\right]
\end{equation*}
for $b\in B_S^\infty$, so that $\delta$ is an inner derivation for each $n\neq0$.
\end{prop}
\begin{proof}
Notice that a derivation $B_S^\infty\to B_S^\infty$ defines (an unbounded) derivation $\mathcal{B}_S\to B_S$, so we can use Proposition \ref{covariant_der_B_S}. In particular, there exists $F\in C(\Z/S\Z)$ such that $\delta(b) = \left[U^nM_{F},b\right]$, at least for $b\in \mathcal{B}_S$. We want to show that in fact  $F\in\mathcal{E}(\Z/S\Z)$. Then, by continuity, that formula will work for any $b\in B_S^\infty$.

Let $\chi$ be a character on $\Z/S\Z$. Then, since $M_\chi\in \mathcal{B}_S\subset B_S^\infty$, we have that $\delta(M_\chi)\in B_S^\infty$ is well-defined.  Using the fact that $\beta^n(\chi)(x)= \chi(x)\chi(q(n))$ for $x\in\Z/S\Z$, we have
\begin{equation*}
\delta(M_\chi) = \left[U^nM_{F},M\chi\right] = U^n\left(M_{F\chi} - M_{F\beta^n(\chi)}\right) = (1-\chi(q(n)))U^nM_{F}M_\chi .
\end{equation*}
However, since $\chi$ is a character on $\Z/S\Z$, we have $\chi\in\mathcal{E}(\Z/S\Z)$ and for each $n$ we can choose $\chi$ such that $\chi(q(n))\neq 1$.  Therefore we obtain:
\begin{equation*}
M_{F} = \frac{1}{1-\chi(q(n))}U^{-n}\delta(M_\chi)M_\chi^{-1}\in B_S^\infty\,.
\end{equation*}
Thus indeed $F\in\mathcal{E}(\Z/S\Z)$, completing the proof.
\end{proof}

In particular, the above proposition can be applied to the $n$-th Fourier components of a continuous derivation $\delta:B_S^\infty \to B_S^\infty$. 

Next, utilizing Proposition \ref{invariant_der_B_S} we classify all $0$-covariant (invariant) derivations on $B_S^\infty$. 
\begin{prop}\label{invariant_der_B_S_infty}
Let $\delta:B_S^\infty\to B_S^\infty$ be an invariant continuous derivation. Then there exists a unique constant $C$ such that 
\begin{equation*}
\delta= C\delta_{\mathbb L} + \tilde{\delta}\,,
\end{equation*} 
where $\tilde{\delta}$ is an inner derivation of the form:
\begin{equation*}
\tilde{\delta}(a)= [M_G,a]
\end{equation*}
for some $G\in\mathcal{E}(\Z/S\Z)$.
\end{prop}

\begin{proof}
Let $\delta:B_S^\infty\to B_S^\infty$ be a continuous invariant derivation. Clearly $\delta$ defines a derivation on $\mathcal{B}_S$. Following the ideas of the proof of Proposition \ref{invariant_der_B_S}, from \cite{KMRSW2}, we have $\delta(U)=UM_F\in B_S^\infty$ for some $F\in\mathcal{E}(\Z/S\Z)$ and $\delta(M_f)=0$ for $f\in\mathcal{E}(\Z/S\Z)$ since there are no nonzero derivations on $\mathcal{E}(\Z/S\Z)$ by invariance.  We decompose $F$ in the following way
\begin{equation*}
F=C+\tilde{F}\,,
\end{equation*}
where $C=\int_{\Z/S\Z}F$ and $\int_{\Z/S\Z}\tilde{F}=0$. Thus we obtain 
\begin{equation*}
\delta= C[\mathbb{L},\cdot] + \tilde{\delta}\,,
\end{equation*}
with $\tilde{\delta}(U)=UM_{\tilde{F}}$.
To complete the proof, we must show $\tilde{\delta}$ is an inner derivation.  That is, we need to show there exists a $G\in\mathcal{E}(\Z/S\Z)$ such that
\begin{equation*}
\tilde{\delta}(a) = [M_G,a]
\end{equation*}
for any $a\in B_S^\infty$.  If that's the case, then
\begin{equation*}
UM_{\tilde{F}} = \tilde{\delta}(U) = [M_G,U] = U(M_{G\circ\beta}-G)
\end{equation*}
which is true if
\begin{equation*}
\tilde{F}(x) = G(x+q(1))-G(x)
\end{equation*}
for some $G$.  However, such a $G$ exists since $\tilde{F}\in\mathcal{E}(\Z/S\Z)$ and hence is a finite linear combination of nontrivial characters on $\Z/S\Z$, as $\int_{\Z/S\Z}\tilde{F}=0$. By continuity, $\chi(q(1))=1$ if and only if $\chi$ is trivial. If $\tilde{F}(x) = \chi(x)$ with $\chi(q(1))\neq1$, then we can choose
\begin{equation*}
G(x) = \frac{\chi(x)}{\chi(q(1))-1}
\end{equation*}
which is clearly in $\mathcal{E}(\Z/S\Z)$.  Thus in general, such a $G$ will also be a finite linear combination of characters and $G\in\mathcal{E}(\Z/S\Z)$.
\end{proof}

\subsection{Classification}
Finally, we classify all continuous derivations $\delta:B_S^\infty\to B_S^\infty$.
\begin{theo}
Let $\delta:B_S^\infty\to B_S^\infty$ be a continuous derivation.  Then there exists a unique constant $C$ such that
\begin{equation*}
\delta = C\delta_{\mathbb{L}} + \tilde\delta
\end{equation*}
where $\tilde\delta$ is inner. 
\end{theo}

\begin{proof}
Let $\delta_0$ be the $0$-th Fourier component of $\delta$. It is an invariant derivation, so by Proposition \ref{invariant_der_B_S} we have the unique decomposition:
\begin{equation*}
\delta_0(b)= C[\mathbb{L},b] + \tilde{\delta_0}(b)
\end{equation*}
for every $b\in B_S^\infty$, where $\tilde{\delta_0}$ is an inner derivation.
From Proposition \ref{cont_covariant_der_B_S_infty} we have that the Fourier components $\delta_n$, $n\neq 0$ are inner derivations.   It follows from Propositions \ref{cont_covariant_der_B_S_infty} and \ref{Fourier_series_norm_conv}, by extracting $\delta_0$, that we have:
\begin{equation*}
\delta(b)=\delta_0(b)+\lim_{n\to\infty}\sum_{|j|\leq n,\, j\neq0}\delta_j(b) = \delta_0(b)+\lim_{n\to\infty}\left[\sum_{|j|\leq n,\, j\neq0}U^jM_{F_j},b\right]
\end{equation*}
with $F_j\in\mathcal{E}(\Z/S\Z)$.  We now need to establish two things:  that $\{F_n\}$ is ULC and that the series $\sum_{j}U^jM_{F_j}$ is convergent in $B_S^\infty$.

Since $\delta:B_S^\infty\to B_S^\infty$ is continuous, from Lemma \ref{reed_simon_1}, we know it's a continuous map between the Fr\'echet spaces $B_{S,\, l}^\infty$ for some $l$.  Thus, given $b\in B_S^\infty$, we have $\delta(b)\in B_S^\infty$ and so there exists a divisor $l$ of $S$ such that $U^l\delta(b)U^{-l} = \delta(b)$.  But, since $\rho_\theta$ is a continuous automorphism, we have $U^l\delta_n(b)U^{-l} = \delta_n(b)$ for any $n$.  Using this observation for $b=U$, and the formula:
\begin{equation*}
\delta_n(U) = U^{n+1}M_{\beta(F_n)-F_n}\,,
\end{equation*}
we see that for some $l|S$ and we have:
\begin{equation*}
\beta^l(\beta(F_n)-F_n) = \beta(F_n)-F_n
\end{equation*}
for every $n$.  This means that
\begin{equation*}
F_n(x+(l+1)q(1)) - F_n(x+lq(1)) = F_n(x+q(1))-F_n(x),
\end{equation*}
and so the difference $F_n(x+lq(1)) - F_n(x)$ is $\beta$-invariant.
By minimality of the map $\beta$, there is a constant $c_n$ such that 
$$F_n(x+lq(1)) = F_n(x) + c_n.$$
Iterating this equation we get:
\begin{equation*}
F_n(x+klq(1)) = F_n(x) + kc_n\,
\end{equation*}
for $k=0,1,2,\ldots$. In the topology of $\Z/S\Z$ we have $klq(1)\to0$.  By taking the limit $k\to\infty$ in the above formula we see that $c_n=0$ for all $n$ and hence $F_n$ is ULC.  

Finally, to see that 
$$\lim_{n\to\infty} \sum_{|j|\leq n,\, j\neq0}U^jM_{F_j}\in B_S^\infty,$$ 
we check that $\{\|F_n\|\}$ are RD.  From the proof of Proposition \ref{cont_covariant_der_B_S_infty}, for $n\neq0$ we have
\begin{equation*}
M_{F_n} = \frac{1}{1-\chi(q(n))}U^{-n}\delta_n(M_\chi)M_\chi^{-1}\in B_S^\infty\,
\end{equation*}
for any character $\chi$ such that $\chi(q(n))\neq1$.  Therefore by Lemma \ref{est_delta_n}, 
\begin{equation*}
\|F_n\| = \|M_{F_n}\| = \|M_{F_n}\|_0\le \frac{1}{|1-\chi(q(n))|}\cdot\frac{C_k}{n^k}
\end{equation*}
for some constants $C_k$.  The point now is to choose a particular character to estimate the denominator in the above formula.

Let $g=\textrm{gcd}(n,S)$ and write $n=g\cdot n'$ with $n'$ and $S$ relatively prime.  Then set $\chi(q(n)) = e^{2\pi inj/l}$ for some divisor $l$ of $S$ to be chosen below, $0\le j<l$ and $j$ relatively prime to $l$.  Then $\chi(q(n))=1$ if and only if $l$ is a divisor of $nj=gn'j$. Take $l=gh$ so that $gh$ is a divisor of $S$.  Then $\chi(q(n)) = e^{2\pi in'j/h}$ and $n'$ is relatively prime to $h$.  Thus, there exist integers $p$ and $q$ so that $pn'+qh=1$.  Take $j=p\gamma$ for some integer $\gamma$ to be determined.  With those choices we obtain:
\begin{equation*}
\chi(q(n)) = e^{2\pi in'p\gamma/h} = e^{2\pi i\gamma/h}\,.
\end{equation*}
Choose $\gamma$ as follows:
\begin{equation*}
\gamma=\left\{
\begin{aligned}
&\frac{h}{2} &&\textrm{if }h\textrm{ is even.}\\
&\frac{h+1}{2} &&\textrm{if }h\textrm{ is odd.}
\end{aligned}\right.
\end{equation*}
Thus, $\chi(q(n))=-1$ for even $h$ and $\chi(q(n))\approx -1$ for odd large enough $h$.  Thus 
$$|1-\chi(q(n))|>3/2$$ for this choice of $\chi$ and therefore we get:
\begin{equation*}
\|F_n\| \le \frac{2C_k}{3n^k}.
\end{equation*}
Hence, $\{\|F_n\|\}$ are RD.
\end{proof}

\subsection{Counterexamples}
It is necessary that the derivation is continuous in order to conclude that $\tilde{\delta}$ is in fact inner.  Here we present a class of examples of derivations $\delta: B_S^\infty\to B_S^\infty$ are not inner modulo $C\delta_{\mathbb{L}}$, and consequently are not continuous.

We will study below examples of derivations on $B_S^\infty$ which are of the form:
\begin{equation*}
\delta_F(a) = [F(U),a]\,,
\end{equation*}
where the function $F(z)$ is not smooth, so that the derivations are not inner.
For $a\in B_S^\infty$ we will use the decomposition found in Equation \ref{CharDecomp}:
\begin{equation*}
a=\sum_{k=0}^{l-1}f_k(U)M_{\chi_l^k}\,.
\end{equation*}
Then
\begin{equation*}
\delta_F(a) = [F(U),a]=\sum_{k=0}^{l-1}f_k(U)\left[F(U)-F\left(e^{2\pi ik/l}U\right)\right]M_{\chi_l^k}\,.
\end{equation*}
Notice that 0-th Fourier component of derivations $\delta_F$ is always zero. To show that $\delta_F$ is not inner, we construct an $F(z)\notin C^\infty(S^1)$ such that for any divisor $l$ of $S$ and $0\le k\le l-1$, $F(z)-F(e^{2\pi ik/l}z)\in C^\infty(S^1)$.  First consider the case when $S=p^\infty$ for some prime $p$ and write $F$ in its Fourier decomposition:
\begin{equation*}
F(z) = \sum_{n\in\Z}a_kz^k
\end{equation*}
for some coefficients $a_k$ to be determined.  Notice that
\begin{equation*}
F(z)-F\left(e^{2\pi ik/l}z\right) = \sum_{n\in\Z}a_k\left(1-e^{\frac{2\pi ik}{p^n}}\right)z^k\,.
\end{equation*}
Define the coefficients $a_k$ as follows:
\begin{equation*}
a_k =\left\{
\begin{aligned}
&b_l && k=p^l\,,\,l\ge0 \\
&0 &&\textrm{else}\,.
\end{aligned}\right.
\end{equation*}
Then we have 
\begin{equation*}
F(z)-F\left(e^{2\pi ik/l}z\right) = \sum_{l=0}^\infty b_l\left(1-e^{\frac{2\pi ip^l}{p^n}}\right)z^{p^l} = \sum_{l=0}^{n-1} b_l\left(1-e^{\frac{2\pi ip^l}{p^n}}\right)z^{p^l}\,,
\end{equation*}
which is clearly in $C^\infty(S^1)$ for every $b_l$.  For example, if we chose $b_l=1$ then $F(z)\notin C^\infty(S^1)$ because its Fourier coefficients are not rapidly decaying.  To extend this to general supernatural numbers $S$, let $1<l_1|\,l_2|\cdots$ be a sequence of increasing divisors of $S$ that converge to $S$.  If $l$ is any divisor of $S$, then there exists a $n$ large enough, such that $l|\,l_n$.  Repeating the above steps we get
\begin{equation*}
F(z) = \sum_{n=0}^\infty a_kz^{l_n}\quad\textrm{and}\quad F(z)-F\left(e^{2\pi ik/l}z\right) = \sum_{n=1}^\infty a_n\left(1-e^{\frac{2\pi il_n}{l}}\right)z^{l_n}\,.
\end{equation*}
The last sum in the above equation will also be finite for the same reasons as the case $S=p^\infty$, and consequently such derivations are well-defined for any sequence $\{a_n\}$.

\section{$K$-Theory and $K$-Homology}
In this section we study the $K$-Theory and $K$-Homology of $B_S$. Though the $K$-Theory of Bunce-Deddens algebras is known, we include an explicit calculation of it for completeness. Using the known $K$-Theory, we invoke Rosenberg and Schochet's Universal Coefficient Theorem \cite{RS UCT}, as well as Pimsner and Voiculescu's  $6$-term exact sequence \cite{PV}, to compute the $K$-Homology. 

\subsection{K-Theory}
For the remainder of the paper, we assume that $S$ is an infinite supernatural number. For a $C^*$-algebra $A$, we denote by $[\cdot]_0$ and $[\cdot]_1$ the class of an element in $K_0(A)$ and $K_1(A)$, respectively. Define the group 
$$G_S : =\left \{ \frac{k}{l} \in \Q : k\in\Z, l | S\right\}.$$ 
The proof of the following proposition is based on lemmas which are stated and proved directly after. 

\begin{prop}
$K_0(B_S) = G_S$, and $K_1(B_S) = \Z$. 
\end{prop}
\begin{proof}
Let $\{l_i\}$ be a sequence of finite divisors of $S$ such that $l_i | l_{i+1}$, $l_i < l_{i+1}$, and $\lim_{i \to \infty} l_i = S$. For each $i$, consider 
\begin{equation*}
    B_S \supseteq B_{S,i} : = \left\{a \in B_S : U^{l_i}aU^{-l_i} = a\right\}.
\end{equation*}
It is easy to see that $B_{S,i}$ can be identified with the Bunce-Deddens algebra for finite number $l_i$. Therefore, from \cite{KMRSW2}, there is an isomorphism of $C^*$-algebras:
$$B_{S,i} \cong C(S^1) \otimes M_{l_i}(\mathbb{C}).$$
Since $l_i | l_{i+1}$, we have that $B_{S,i} \subseteq B_{S,i+1}$. Moreover, $\cup_{i=1}^\infty B_{S,i}$ is dense in $B_S$, and we can realize $B_S$ as an inductive limit of $C^*$-algberas
\begin{equation*}
    B_S = \overline{\cup_{i=1}^\infty B_{S,i}}. 
\end{equation*}
By the stability of $K_0$ and $K_1$ we have: 
$$K_0(B_{S,i}) \cong K_0(C(S^1)) = \mathbb{Z},$$
and 
$$K_1(B_{S,i}) \cong K_1(C(S^1)) = \mathbb{Z}.$$  
By the continuity of $K_0$ and $K_1$, it only remains to compute the corresponding direct limits. By Lemma \ref{K_0_gen}, for each $j$ we have that $K_0(B_{S,i})$ is generated by the class of the projection $M_{\kappa_{l_i,j}}$, where 
\begin{equation}\label{chi_def}
    \kappa_{l,j}(x) : =\left\{
\begin{aligned}
& 1 &&\textrm{if }l | (l-j)\\
& 0 &&\textrm{else}
\end{aligned}\right.
\end{equation}
 Let $\iota_i: B_{S,i} \hookrightarrow B_{S,i+1}$ denote the inclusion. By Lemma \ref{K_0_induced}, the induced map $K_0(\iota_i):K_0(B_{S,i}) \to K_0(B_{S,i+1})$ on generators is given by 
\begin{equation*}
    K_0(\iota_i)\left[M_{\kappa_{l_i,0}}\right]_0 = \frac{l_{i+1}}{l_i} \left[M_{\kappa_{l_{i+1},0}}\right]_0.
\end{equation*}
From this, we see $K_0(B_S)$ is isomorphic to the space of equivalence classes on the disjoint union of copies of $\Z$:
$$K_0(B_S)\cong \left(\bigsqcup_{i=1}^\infty \Z \right)/ \sim$$ 
where $x_i \sim x_j \iff \textrm{ there exist } r \geq t, k \geq s \textrm{ such that } \frac{l_r}{l_t} x_i =  \frac{l_k}{l_s} x_j$. 

The key trick to compute the inductive limit above is to identify $K_0(B_{S,i})$ with the subset of rational numbers $\{k/l_i : k \in \Z \} \cong \Z$. From the fact that 
$$\frac{k}{l_i} = \frac{k'}{l_{i+1}} \iff \frac{kl_{i+1}}{l_i} = k',$$
we see that the inductive limit takes the form 
$$G'_S = \left\{ \frac{k}{l_i} : i = 1,2,\dots \right\} \subseteq G_S.$$ 
To see the reverse inclusion $G_S \subseteq G_S'$, fix $\frac{k}{l} \in G_S$. Since $\lim_{i \to \infty} l_i = S$, there exists $i$ such that $l | l_i$. For such an $i$ we have: 
$$\frac{k}{l} = \frac{k (l_i / l)}{l_i} \in G_S'.$$  
Therefore, $K_0(B_S)$ is isomorphic to $G_S$.

To compute $K_1(B_S)$, the approach is similar. In this case, however, the inclusion map $\iota_i$ induces the identity map between $K_1(B_{S,i})$ and $K_1(B_{S,i+1})$. This is verified in Lemma \ref{K_1_calc}. Hence, the direct limit is simply $\mathbb{Z}$. This completes the proof. 
\end{proof}
 
We now proceed by proving the lemmas used in the above theorem.
\begin{lem}
\label{K_0_gen}
For every $j$, the class $[M_{\kappa_{l_i,j}}]_0$ generates $K_{0}(B_{S,i})$.  
\end{lem}
\begin{proof}
Notice that, under the isomorphism $B_{S,i} \cong C(S^1) \otimes M_{l_i}(\mathbb{C})$ mentioned above, $M_{\kappa_{l_i,j}}$ corresponds to an element of the form $ 1 \otimes E_{j,j}$, where $E_{j,j}$ is a rank $1$ projection. Recall the well known facts that $[1]_0$ generates $K_0(C(S^1))$, and that the map $ f \mapsto f \otimes E_{j,j} $ induces an isomorphism between $K_0(C(S^1))$ and $K_0(C(S^1) \otimes M_{l_i}(\C))$ (see, for instance, \cite{RLL} Exercise $11.2$ and Proposition 4.3.8). From this it is clear that $[1 \otimes E_{j,j}]_0$ generates $K_0(C(S^1)\otimes M_{l_i}(\C))$, and therefore that  $[M_{\kappa_{l_i,j}}]_0$ generates $K_0(B_{S,i})$. 
\end{proof} 
 
\begin{lem}
\label{K_0_induced}
The induced map $K_0(\iota_i):K_0(B_{S,i}) \to K_0(B_{S,i+1})$ is given by 
\begin{equation*}
    K_0(\iota_i)\left[M_{\kappa_{l_i,0}}\right]_0 = \frac{l_{i+1}}{l_i} \left[M_{\kappa_{l_{i+1}},0}\right]_0
\end{equation*}
\end{lem}
\begin{proof}
Notice that $l_i | x$ if and only if $l_{i+1} | x$ or $l_{i+1} | (x-l_i)$ or $l_{i+1} | (x - 2l_i)$ or \dots or, finally, $l_{i+1} | \left(x - l_{i+1}-l_i \right) $. In other words, we have the formula 
\begin{equation*}
    M_{\kappa_{l_i,0}} = M_{\kappa_{l_{i+1},0}} + M_{\kappa_{l_{i+1}, l_i}} + M_{\kappa_{l_{i+1}, 2l_i}} + \ldots + M_{\kappa_{l_{i+1},\left(\frac{l_{i+1}}{l_i} - 1 \right)l_i}}.
\end{equation*}
Since $\left[M_{\kappa_{l_i,0}}\right]_0$ generates $K_0(B_{S,i})$ and for any $j$, $\left[M_{\kappa_{l_{i+1},j}}\right]_0$ generates $K_0(B_{S,i+1})$, we have the following calculation:
\begin{equation*}
    K_0(\iota_i)\left[M_{\kappa_{l_i,0}}\right]_0 = \left[M_{\kappa_{l_{i+1},0}} + M_{\kappa_{l_{i+1}, l_i}} + M_{\kappa_{l_{i+1}, 2l_i}} + \ldots + M_{\kappa_{l_{i+1},\left(\frac{l_{i+1}}{l_i} - 1 \right)l_i}}\right]_0
\end{equation*}
\begin{equation*}
 = \left[M_{\kappa_{l_{i+1},0}}\right]_0 + \left[M_{\kappa_{l_{i+1}, l_i}}\right]_0 + \left[M_{\kappa_{l_{i+1}, 2l_i}}\right]_0 + \ldots + \left[M_{\kappa_{l_{i+1},l_{i+1} - l_i}}\right]_0 = \frac{l_{i+1}}{l_i}\left[M_{\kappa_{l_{i+1},0}}\right]_0 
\end{equation*}
where we were able to separate the class of the sum into the sum of classes since $\{M_{\kappa_{l_{i+1},j}}\}$ are mutually orthogonal.
\end{proof} 
 
\begin{lem}
\label{K_1_calc}
The inclusion map $\iota_i:B_{S,i} \to B_{S,i+1}$ induces the identity map between $K_1(B_{S,i})$ and $K_1(B_{S,i+1})$.
\end{lem}
\begin{proof}
Denote by $z\in C(S^1)$ the map $S^1\ni z \mapsto z\in\C$. Recall the well known facts that $[z]_1$ generates $K_1(C(S^1))$, and that the map $ f \mapsto f \otimes E_{0,0} $ induces an isomorphism between $K_1(C(S^1))$ and $K_1(C(S^1) \otimes M_{l_i}(\C)) \cong K_1(M_{l_i}(C(S^1)))$ (see \cite{RLL} Example 11.3.4 and \cite{WO} Lemma 7.1.8). Since both algebras are unital, it is easily verified that under the above map $ f \mapsto f \otimes E_{0,0}$ we have the following correspondence:
\begin{equation*}
    [z]_1 \mapsto \left[ \begin{pmatrix}
    z & 0 \\
    0 & 1_{l_{i}-1}
    \end{pmatrix} \right]_1
\end{equation*}
(see \cite{RLL} exercise $8.5$). We can easily adapt Lemma 2.1.5 in \cite{RLL} to see that by homotopy the above class is the same as:
\begin{equation*}
    \left[ \begin{pmatrix}
    z^{1/l_i} & 0 & 0 & \cdots & 0 \\
    0 & z^{1/l_i} & 0 & \cdots & 0\\
    \vdots & \vdots & \vdots & \ddots & \vdots \\
    0 & 0 & 0 & \cdots & z^{1/l_i}
    \end{pmatrix} \right]_1.
\end{equation*}
Under the identification between $B_{S,i}$ and $C(S^1) \otimes M_{l_{i}}(\C)$, we see that $z^{1/l_i} \otimes 1_{l_i}$ corresponds to $U$. Since $\iota_i(U) = U$, we have that $\iota_i$ induces the identity map between $K_1(B_{S,i})$ and $K_1(B_{S,i+1})$.  
\end{proof}

\subsection{K-Homology and UCT}
Before computing the $K$-Homology of $B_S$, we need the following lemma describing homomorphisms from $K_0(B_S)$ to $\Z$.
\begin{lem}\label{HomLemma}
There are no nontrivial homomorphisms from $G_S$ to $\mathbb{Z}$. 
\end{lem}
\begin{proof}
Let $\phi: G_S \to \Z$ be a homomorphism. For each divisor $l$ of $S$, consider the subgroup 
\begin{equation*}
    G_l = \left\{\frac{k}{l} : k \in \Z\right\} \cong \Z. 
\end{equation*}
$\phi$ therefore must restrict to a homomorphism $\phi_l : G_l \to \Z$. Such a homomorphism is clearly determined by its action on $\frac{1}{l}$. Denote \begin{equation*}
    \phi \left(\frac{1}{l}\right) = \phi_l\left(\frac{1}{l}\right) : = a_l.
\end{equation*}
If $l'$ is another divisor of $S$ such that $l | l'$, clearly $G_l \subseteq G_{l'}$. Since $\frac{k(l'/l)}{l'} = \frac{k}{l}$, we have that 
\begin{equation}
\label{divisor}
    \left(l'/l \right) a_{l'} = a_l.
\end{equation}
With this in mind, let $l$ be any finite divisor of $S$, and let $\{l_i\}_{i=1}^\infty$ be a sequence such that $l_i | l_{i+1}$, $l_{i} < l_{i+1}$, $\lim_i l_i = S$, and $l_1 = l$. Suppose that $a_l \neq 0$. Then $a_l \in \Z$ is some finite integer such that $a_l$ is divisible by $(l_i/l)$ for every $i$. Since $l$ is finite and $\lim_i l_i$ is an infinite supernatural number, this is clearly a contradiction. Hence, $a_l = 0$ for every finite divisor $l$. Equation \ref{divisor} shows this holds for any divisor. 
\end{proof}
We are now ready to compute the $K$-homology of $B_S$. We will use the identification $KK^i(A,\mathbb{C}) : = K^i(A)$. Since $B_{S}$ is an inductive limit of $C(S^1) \otimes M_{l_i}(\mathbb{C})$, 
it is clear that $B_S$ falls into the bootstrap category (see \cite{BB}, page 228), and therefore the conclusion of the Universal Coefficient Theorem of Rosenberg and Schochet holds. Using this fact, the calculation of $K^0(B_S)$ is simple.
\begin{prop}
The K-Homology groups of $B_S$ are: 
$$K^0(B_S) = 0 \textrm{ and }K^1(B_S) = \mathbb{Z} \oplus \left( \left(\mathbb{Z}/S\mathbb{Z}\right)/\mathbb{Z} \right).$$
\end{prop}
\begin{proof}
The Universal Coefficient Theorem states that we have an exact sequence: 
\begin{equation*}
    \begin{tikzcd}
    0 \arrow{r} & \Ext^1_\Z(K_1(B_S), \Z) \arrow{r} & K^0(B_S) \arrow{r} & \Hom(K_0(B_S),\Z) \arrow{r} & 0,
    \end{tikzcd}
\end{equation*}
We direct readers unfamiliar with computations involving $\Ext_\Z^1$ to the appendix. Since $\Ext^1_\Z(K_1(B_S), \Z) \cong 0$, we have an isomorphism $K^0(B_S) \cong \Hom(K_0(B_S),\Z)$. The latter group is isomorphic to $0$ by the preceding Lemma. This proves the first claim. For $K^1(B_S)$, again from the Universal Coefficient Theorem, we have an exact sequence 
\begin{equation*}
    \begin{tikzcd}
    0 \arrow{r} & \Ext^1_\Z(K_0(B_S), \Z) \arrow{r} & K^1(B_S) \arrow{r} & \Hom(K_1(B_S),\Z) \arrow{r} & 0,
    \end{tikzcd}
\end{equation*}
which splits unnaturally \cite{RS}. From the final example in the Appendix we have: 
$$\Ext^1_\Z(K_0(B_S), \Z) \cong \left(\mathbb{Z}/S\mathbb{Z}\right)/\mathbb{Z}.$$ 
Hence, we obtain $K^1(B_S) \cong \Z \oplus \left( \Z / S\Z \right)/ \Z $, by the splitting of the sequence.  
\end{proof}

\subsection{K-Homology and PV}
One pitfall of the above calculation is that we obtain no information regarding which subgroup of $\mathbb{Z}/S\mathbb{Z}$ was quotiented out. We amend this through the following explicit calculation that uses different techniques that have merit on their own. It turns out that the subgroup is the natural dense subgroup identified in Equation \ref{q_def}.

Recall that the Pimsner-Voiculescu $6$-term exact sequence for $K$-Homology of the crossed product $C(\Z/S\Z)\rtimes_\beta \Z$ reads:
\begin{equation*}
\begin{tikzcd}
 K^0(C(\mathbb{Z}/S\mathbb{Z})) \arrow{d}[swap]{\partial^1}   & K^0(C(\mathbb{Z}/S\mathbb{Z}))  \arrow{l}[swap]{1-\beta^*} & K^0(B_S) \arrow{l}[swap]{\iota^*}    \\
 K^1(B_S) \arrow{r}{\iota^*} & K^1(C(\mathbb{Z}/S\mathbb{Z})) \arrow{r}{1 - \beta^*} & K^1(C(\mathbb{Z}/S\mathbb{Z})) \arrow{u}{\partial^0}
\end{tikzcd}
\end{equation*}
The original formulation can be found in \cite{PV}. The notation $\beta^*$ we used above is for the map induced in K-Homology by $\beta$ from Equation \ref{beta_def}.

Since $\mathbb{Z} / S \mathbb{Z}$ is totally disconnected, it follows from Exercise 3.4 in \cite{RLL} that
$$K_0(C(\mathbb{Z}/S\mathbb{Z})) \cong C(\mathbb{Z}/S\mathbb{Z},\mathbb{Z})=\mathcal{E}(\Z / S\Z, \Z).$$ 
The last equality above follows because $\Z$ is discrete. 
Additionally, from \cite{KD} Example III.2.5, $C(\Z / S \Z)$ is AF.
Therefore, by Exercise 8.7 in \cite{RLL}, we have:
$$K_1(C(\Z / S\Z)) = 0.$$
Since $\mathcal{E}(\Z / S\Z, \Z)$ is free (see exercise 7.7.5 in \cite{HR}), using the Universal Coefficient Theorem we obtain 
$$K^1(C(\Z / S \Z)) = 0 \textrm{ and } K^0(C(\Z / S \Z)) = \Hom(\mathcal{E}(\Z / S\Z, \Z),\Z).$$
Putting together these identifications, we can rewrite the above exact sequence as follows:
\begin{equation*}
\begin{tikzcd}
0 \arrow{r}{\iota^*} & \Hom(\mathcal{E}(\Z / S\Z, \Z),\Z)  \arrow{r}{1-\beta^*} & \Hom(\mathcal{E}(\Z / S\Z, \Z),\Z) \arrow{r}{\partial^1} & K^1(B_S) \arrow{r} & 0. 
\end{tikzcd}
\end{equation*}
From exactness, we conclude the following identification:
\begin{equation*}
   K^1(B_S) \cong \Hom(\mathcal{E}(\Z / S\Z, \Z),\Z) / (\textrm{Im}(1 - \beta^*)).
\end{equation*}

In what follows we compute the range of $1 - \beta^*$. To do that we consider a more geometrically motivated identification of $\Hom(\mathcal{E}(\Z / S\Z, \Z),\Z)$, outlined presently. 

Borrowing from \cite{KMR}, given a supernatural number $S$, we consider a set:
$$V_S = \{(l,k) : l | S \textrm{ and } 0 \leq k < l \}.$$ 
We define the following subset of functions on $V_S$ with values in $\Z$: 
\begin{equation*}
    \Phi : = \left\{ \phi:V_S \to \mathbb{Z} : \phi(l,k) = \sum_{j=0}^{(l'/l) - 1} \phi(l', k+jl) \textrm{ whenever } l | l' \right\}.
\end{equation*} 
The set $\Phi$ is an Abelian group with respect to addition of functions.
We claim that $\Phi$ is naturally isomorphic to the group $\Hom(\mathcal{E}(\Z / S\Z, \Z),\Z)$ through the following identification. 

To each $\phi \in \Phi$, we associate a homomorphism $\psi_\phi \in \Hom(\mathcal{E}(\Z / S\Z, \Z),\Z)$ defined on characteristic functions $\kappa_{l,k}$, defined in Equation \ref{chi_def}, by  

\begin{equation*}
    \psi_\phi (\kappa_{l,k}) = \phi(l,k).
\end{equation*}
 Since for $f \in \mathcal{E}(\Z / S \Z,\Z)$, there exists $l$ such that $f(x + l ) = f(x)$ for all $x$, we can write 
\begin{equation*}
    f(x) = \sum_{0 \leq k < l} f(k) \kappa_{l,k}(x).
\end{equation*}
Therefore, any $\psi \in \Hom(\mathcal{E}(\Z / S\Z, \Z),\Z)$ is completely determined by the values $\psi(\kappa_{l,k})$. Due to the equality 
\begin{equation*}
    \kappa_{l,k}(x) = \sum_{j=0}^{(l'/l) - 1} \kappa_{l',k+jl}(x)
\end{equation*}
for any $l'$ such that $l | l'$, it is clear that any homomorphism $\psi$ can be associated to a unique element in $\Phi$. Hence, we have established the following proposition. 

\begin{prop} With above notation we have a group isomorphism:
$$\Phi \cong \Hom(\mathcal{E}(\Z / S\Z, \Z),\Z).$$
\end{prop}
The goal now is to compute the range of $1-\beta^*$ under this identification. Our method will be to define a surjective map $\Phi \to \Z \oplus (\Z / S\Z )/\Z$ whose kernel is precisely $\Im(1 - \beta^*)$. To define such a map, we need the following observation. 

For $l,l'$ such that $l|l'|S$ and $\phi\in \Phi$, we define an integer
\begin{equation*}
    R\phi(l,l') : = \sum_{a = 1}^{(l'/l)-1} \sum_{j=0}^{al-1} \phi(l',j).
\end{equation*}
By convention, a summation from bigger to smaller index value is set to be zero, so that $R\phi(l,l)=0$. The quantity $R\phi(l,l')$ above has the following key property.
\begin{lem}
\label{consistency_condition} With $l,l',\phi$ as above we have:
$R\phi(1,l') - R\phi(1,l) = l R\phi(l,l').$
\end{lem}
\begin{proof}
We begin with the left-hand-side. Note first that in general we have the following summation decomposition for any function $f$:
$$\sum_{i=0}^{l'-1} f(i) = \sum_{b=0}^{(l'/l)-1} \sum_{j=0}^{l-1} f(bl+j).$$ 
Using this observation, we compute:
\begin{equation*}
   R\phi(1,l') - R\phi(1,l)  = \sum_{i=1}^{l'-1} \left(\sum_{j=0}^{i-1} \phi(l',j) \right) - \sum_{a=1}^{l-1} \sum_{j=0}^{a-1} \phi(l,j) = 
\end{equation*}
\begin{equation*}
   \sum_{b=0}^{(l'/l)-1} \sum_{a=0}^{l-1} \sum_{j=0}^{bl+a-1} \phi(l',j)  - \sum_{b=0}^{(l'/l)-1} \sum_{a=0}^{l-1} \sum_{j=bl}^{bl+a-1} \phi(l',j) = \sum_{a=0}^{l-1} \sum_{b=0}^{(l'/l)-1}  \sum_{j=bl}^{bl+a-1} \phi(l',j)
\end{equation*}
\begin{equation*}
    = l \left( \sum_{b=0}^{(l'/l)-1} \sum_{j=0}^{bl-1} \phi(l',j) \right) = lR\phi(l,l').
\end{equation*}
\end{proof}
As a corollary, we obtain the following congruence relation.
\begin{cor}
\label{R_cor}
$R\phi(1,l) \equiv R\phi(1,l')\  \textrm{mod}(l)$
\end{cor}
For convenience, we recall the definition of $\Z / S \Z$ as an inverse limit:  
$$\Z / S \Z = \left\{(x_l)_{l|S} : x_l \in \Z / l \Z \textrm{ and } x_l \equiv x_{l'}\textrm{ mod}(l )\textrm{ when } l|l'\right\}.$$ 
With this in mind, we define the following two group homomorphisms:
\begin{equation*}
    \tau: \Phi \to \Z \textrm{,} \quad \tau (\phi) := \phi(1,0)
\end{equation*}

\begin{equation*}
    \rho: \Phi \to \Z / S \Z \textrm{,} \quad \rho (\phi) := (R\phi(1,l))_{l|S}.
\end{equation*}
Due to Corollary \ref{R_cor}, the latter map is well defined. 

For our final result, we need to check that $\rho$ is surjective. This is verified in the following proposition. 
\begin{prop}
\label{rho_is_onto}
The homomorphism $\rho: \Phi \to \Z / S \Z$ is surjective. 
\end{prop}
\begin{proof}
We begin by rewriting $R\phi(1,l)$ more conveniently as 
\begin{equation}\label{R_formula}
    R\phi(1,l) = \sum_{j=0}^{l-2}(j+1)\phi(l,j).
\end{equation}
We consider a sequence of divisors $\{l_i \}_{i=1}^\infty$ such that $l_i | l_{i+1}$, $l_i < l_{i+1}$, and $\lim_i l_i = S$. Then, similarly to the usual expansion of a p-adic integer, we can represent an arbitrary element $x$ in $\mathbb{Z} / S \Z$ by 
\begin{equation*}
x=    \sum_{n=1}^\infty a_{n} l_{n-1} 
\end{equation*}
with $l_0 = 1$ and ``digits" $0\leq a_n\leq l_n/l_{n-1}$.   
Given $x\in \mathbb{Z} / S \Z$ we use the above expansion to define $\phi: V_S \to \Z$ in the following way: 
\begin{equation*}
    \phi(l_n,0) = a_1 - a_2 - \ldots - a_n.
\end{equation*}
\begin{equation*}
    \phi(l_n,l_k) = a_{k+1} \qquad k = 1, 2, \ldots , n-1.
\end{equation*}
\begin{equation*}
    \phi(l_n,j) = 0 \qquad \textrm{ otherwise. }
\end{equation*}
It is not difficult to check that $\phi$ indeed defines a function in $\Phi$. Moreover, we have:
\begin{equation*}
    R\phi(1,l_n) = \sum_{j=0}^{l_n-2} (j+1) \phi(l_n,j) = a_1 - a_2 - \ldots - a_n + \sum_{i=1}^{n-1} (l_i + 1) a_{i+1}
\end{equation*}
\begin{equation*}
    = a_1 + a_2l_1 + \ldots a_n l_{n-1} \equiv \left(\sum_{j=1}^\infty a_{j} l_{j-1} \right)\ \textrm{mod}(l_n).
\end{equation*}
This completes the proof. 
\end{proof}
We are now ready to explicitly describe the K-Homology group $K^1(B_S)$. Consider the natural dense subgroup $\mathbb{Z} \subset \Z /S \Z$, described in Section 2.  Quotienting out by this subgroup, we obtain a map $\tilde{\rho}: \Phi \to (\Z /S \Z)/ \Z $ defined by 
\begin{equation*}
    \tilde{\rho}(\phi) : = \rho(\phi) + \Z. 
\end{equation*}
This, along with the map $\tau: \Phi \to \Z$ defined above, leads us to the final result of this section. 
\begin{theo}
With the above notation we have: $\textrm{Im}(1 - \beta^*) = \textrm{Ker}(\tau \oplus \tilde{\rho})$. Consequently, we have isomorphisms: 
$$K^1(B_S) \cong \Phi / \textrm{ Ker }(\tau \oplus \tilde{\rho})\cong     \Z \oplus \left( (\mathbb{Z}/S \mathbb{Z}) / \Z \right).$$
\end{theo}
\begin{proof}
It is easily verified that 
$$(\beta^*\phi)(l,k) = \phi(l, k+1)\textrm{ mod}(l)).$$
Note that, in addition, it is clear that $\phi \in \textrm{ Ker }(\tau \oplus \tilde{\rho})$ if and only if $\tau(\phi) = 0$ and there exists some number $\psi(1,0) \in \Z$ with 
$$\rho(\phi) + \psi(1,0) = 0$$ in $\mathbb{Z}/ S \mathbb{Z}$. The last equation is equivalent to the following congruences for each divisor $l|S$:
$$\rho(\phi) + \psi(1,0) \equiv 0 \textrm{ mod}(l).$$  

Letting $\phi = (1- \beta^*)\psi$, we have that 
\begin{equation*}
    \tau(\phi) = \psi(1,0) - \psi(1,1 \textrm{ mod}(1)) = 0.
\end{equation*}
Notice additionally that we have the following formula:
\begin{equation*}
    R\phi(1,l) = \left(\sum_{a=1}^{l-1} \sum_{j=0}^{a-1} \phi(l,j) \right)  = \left(\sum_{a=1}^{l-1} \sum_{j=0}^{a-1} \psi(l,j) - \psi(l,j+1) \right) = 
\end{equation*}
\begin{equation*}
    (l-1)(\psi(l,0)-\psi(l,1)) + (l-2)(\psi(l,1) - \psi(l,2)) + \dots
\end{equation*}
\begin{equation*}
    \ldots + 2(\psi(l,l-3)-\psi(l,l-2) + (\psi(l-2)-\psi(l,l-1))
\end{equation*}
\begin{equation*}
    = l\psi(l,0) - \sum_{j=0}^{l-1}\psi(l,j) = l\psi(l,0) - \psi(1,0).
\end{equation*}
This establishes that $ \textrm{Im}(1 - \beta^*) \subseteq \textrm{Ker}(\tau \oplus \tilde{\rho})$. 

For the other direction, let $(\tau \oplus \tilde{\rho}) \phi = 0$. We define $\psi \in \Phi$ such that $(1-\beta^*)\psi = \phi$ in the following way. First, let $\psi(1,0)$ be the number in $\Z/l\Z$ such that 
$$\rho(\phi) + \psi(1,0) \equiv 0 \textrm{ mod}(l).$$ 
For each $l$, let 
\begin{equation}\label{psi_def}
\psi(l,0) = \frac{R\phi(1,l) + \psi(1,0)}{l}. 
\end{equation}
We can then inductively define $\psi(l,k)$ for all other values of $k$ by 
\begin{equation*}
  \psi(l,k) = \psi(l,0) - \sum_{j=0}^{k-1} \phi(l,j).
\end{equation*}
By construction, it is clear that we have:
$$\phi(l,k) = \psi(l,k) - \psi(l,(k+1) \textrm{ mod}(l)).$$
It therefore only remains to check that $\psi$ represents an element in $\Phi$. i.e., we check that 
\begin{equation}
\label{inPhi}
    \psi(l,k) = \sum_{b=0}^{(l'/l) -1} \psi(l', k + bl).
\end{equation}
Notice that the desired equality holds if and only if 
\begin{equation*}
    \psi(l,0) - \sum_{j=0}^{k-1} \phi(l,j) = \frac{l'}{l} \psi(l',0) - \sum_{b=0}^{(l'/l)-1} \sum_{j=0}^{k+bl-1} \phi(l',j).
\end{equation*}
Using that $\phi(l,j) = \sum_{b=0}^{(l'/l)-1} \phi(l',j+bl)$ and cancelling, this is equivalent to 
\begin{equation*}
    \psi(l,0) = \frac{l'}{l}\psi(l',0) - R\phi(l,l').
\end{equation*}
Rearranging, and using Equation \ref{psi_def} we see that  Equation \ref{inPhi} holds if and only if 
\begin{equation*}
    l R\phi(l,l') = R\phi(1,l') - R\phi(1,l) 
\end{equation*}
which was proved in Lemma \ref{consistency_condition}. 
Hence, we obtain:
\begin{equation}
\label{kernel_image}
    \textrm{Ker}(\tau \oplus \tilde{\rho}) = \textrm{Im}(1 - \beta^*). 
\end{equation}

Since $\rho$ is surjective, by Lemma \ref{rho_is_onto}, and by Equation \ref{R_formula}, $R\phi(1,l)$ does not depend on $\phi(1,0)$, it is clear that $\tau \oplus \tilde{\rho}$ is surjective. Hence, we have the following short exact sequence: 
\begin{equation*}
    \begin{tikzcd}
    0 \arrow{r} & \textrm{Ker}(\tau \oplus \tilde{\rho}) \arrow{r}{\iota} & \Phi \arrow{r}{\tau \oplus \tilde{\rho}} & \Z \oplus \left( (\mathbb{Z}/S \mathbb{Z}) / \Z \right)  \arrow{r} & 0.
    \end{tikzcd}
\end{equation*}
Pairing this information with Equation \ref{kernel_image}, as well as the Pimsner-Voiculescu $6$-term exact sequence, we obtain the following sequence of isomorphisms: 
\begin{equation*}
    \Z \oplus \left( (\mathbb{Z}/S \mathbb{Z}) / \Z \right)  \cong \Phi / \textrm{Ker}(\tau \oplus \tilde{\rho}) \cong \Phi / \textrm{Im}(1-\beta^*) \cong K^1(B_S).
\end{equation*}
This completes the proof. 
\end{proof}

\section{Appendix}
In this Appendix, we give a brief description of $\textrm{Ext}^i_{\mathbb{Z}}(G, \mathbb{Z})$, where $G$ is any Abelian group. There are two equivalent definitions of $\textrm{Ext}^i_\mathbb{Z}(G,\mathbb{Z})$ that are useful to us. The first is in terms of free resolutions of the group $G$.

\begin{defin}
Any Abelian group $G$ admits a resolution of the form of a short exact sequence:
\begin{equation*}
    \begin{tikzcd}
0 \arrow{r} & G_1  \arrow{r} & G_2 \arrow{r} & G \arrow{r} & 0 ,
\end{tikzcd}
\end{equation*}
where $G_1$ and $G_2$ are free Abelian. We can then form the corresponding co-chain complex 
\begin{equation*}
    \begin{tikzcd}
0 \arrow{r} & \textrm{Hom}(G_2 ,\mathbb{Z})  \arrow{r}{g} & \textrm{Hom}(G_1, \mathbb{Z} )  \arrow{r} & 0 .
\end{tikzcd}
\end{equation*}
$\textrm{Ext}^i_{\mathbb{Z}}(G, \mathbb{Z})$ is defined to be the cohomology in the $i$-th position of this complex. Namely,
\begin{equation*}
    \textrm{Ext}_\mathbb{Z}^0 (G, \mathbb{Z}) : =  \textrm{Ker}( g) = \textrm{Hom}(G_2, \mathbb{Z}) \quad \textrm{ and } \quad \textrm{Ext}_\mathbb{Z}^1(G, \mathbb{Z}) = \textrm{Hom}(G_1, \mathbb{Z}) / \textrm{Im}(g) ,
\end{equation*}
while $\textrm{Ext}_\mathbb{Z}^i(G, \mathbb{Z}) = 0$ for all $i \geq 2$. 
\end{defin}

We also have the following equivalent definition.

\begin{defin}
Given an injective resolution of $\mathbb{Z}$ of the form
\begin{equation*}
    \begin{tikzcd}
0 \arrow{r} & \mathbb{Z}  \arrow{r} & I_1 \arrow{r} & I_2 \arrow{r} & 0 ,
\end{tikzcd}
\end{equation*}
where $I_1$ and $I_2$ are divisible Abelian,
we can form the following co-chain complex
\begin{equation*}
    \begin{tikzcd}
0 \arrow{r} & \textrm{Hom}(G, I_1)  \arrow{r}{f} & \textrm{Hom}(G, I_2 )  \arrow{r} & 0 .
\end{tikzcd}
\end{equation*}
We define $\textrm{Ext}_\mathbb{Z}^i(G, \mathbb{Z})$ to be the i-th cohomology of the above complex, that is, 
\begin{equation*}
    \textrm{Ext}_\mathbb{Z}^0 (G, \mathbb{Z}) : =  \textrm{Ker}( f) = \textrm{Hom}(G, I_1) \quad \textrm{ and } \quad \textrm{Ext}_\mathbb{Z}^1(G, \mathbb{Z}) = \textrm{Hom}(G, I_2) / \textrm{Im}(f).
\end{equation*}
\end{defin}
For convenience we now state some basic properties of $\textrm{Ext}^1_\mathbb{Z}(\cdot, \mathbb{Z})$. Proofs of the properties contained in the following proposition can be found in \cite{weibel}. 
\begin{prop} 
\label{homalg}
1)  Ext$^1_\mathbb{Z}(G,\mathbb{Z}) = 0 $ whenever $G$ is free Abelian. 

2) Given a short exact sequence $0 \to K \to G \to H \to 0$, one obtains the following exact sequence: 
\begin{center}
\begin{tikzcd}
    0 \arrow{r} & \Hom(H,\mathbb{Z}) \arrow{r} & \Hom(G,\mathbb{Z}) \arrow{r} & \Hom(K,\mathbb{Z}) 
\end{tikzcd}\end{center} 
\begin{center}
\begin{tikzcd}
   \arrow{r} & \Ext^1_\mathbb{Z}(H,\mathbb{Z})  
   \arrow{r} & \Ext^1_\mathbb{Z}(G,\mathbb{Z}) \arrow{r} & \Ext^1_\mathbb{Z}(K,\mathbb{Z}) \arrow{r} & 0
\end{tikzcd}
\end{center}

3) $\textrm{Ext}^1_\mathbb{Z} \left(\bigoplus_i G_i, \mathbb{Z} \right) \cong \prod_i \textrm{Ext}(G_i,\mathbb{Z})$.
\end{prop}
Now we present some example calculations of $\textrm{Ext}^1_\Z(G,\Z)$ for groups $G$ relevant for this paper.

\begin{example}
$\textrm{Ext}^1_\mathbb{Z}(\mathbb{Z} /n \mathbb{Z}, \mathbb{Z}) \cong \mathbb{Z} / n \mathbb{Z}$.
\end{example}
\begin{proof}
We have the following free resolution:
\begin{equation*}
    \begin{tikzcd}
0 \arrow{r} & \mathbb{Z} \arrow{r}{n} & \mathbb{Z}  \arrow{r} & \mathbb{Z}/n\mathbb{Z} \arrow{r} & 0.
\end{tikzcd}
\end{equation*}
Dualizing and identifying $\textrm{Hom}(\mathbb{Z}, \mathbb{Z}) \cong \mathbb{Z}$, we obtain 
\begin{equation*}
\begin{tikzcd}
0 & \arrow{l}[swap]{\phi}  \mathbb{Z}  & \arrow{l}[swap]{n} \mathbb{Z}  & \arrow{l}  0,
\end{tikzcd}
\end{equation*}
and we see that we have: 
$$\textrm{Ext}^1_\mathbb{Z}(\mathbb{Z} /n \mathbb{Z}, \mathbb{Z}) = \textrm{Ker} (\phi) / \textrm{Im} (n) \cong \Z / n \Z.$$
\end{proof}

\begin{example}
$ \textrm{Ext}^1_\mathbb{Z} (\widehat{\mathbb{Z} / S \mathbb{Z}}, \mathbb{Z}) \cong \mathbb{Z} / S\mathbb{Z}$, where $\widehat{\mathbb{Z} / S \mathbb{Z}}$ is the Pontryagin dual of $\mathbb{Z} / S \mathbb{Z}$. 
\end{example}
\begin{proof}
It is easy to see \cite{M} that we can make the identification: 
$$\widehat{\mathbb{Z}/S\mathbb{Z}} \cong \{z \in S^1 : z^l = 1 \textrm{ for some } l | S \textrm{ } \}.$$ 
In particular, this is a pure torsion group and we therefore have that $\textrm{Hom}(\widehat{\mathbb{Z}/S\mathbb{Z}}, \mathbb{Z}) \cong 0$. From the second definition of $\textrm{Ext}^1_\mathbb{Z}(\cdot, \mathbb{Z})$ with $I_1 = \Q$, $I_2 = \Q / \mathbb{Z}$,
\begin{equation*}
    \textrm{Ext}^1_\mathbb{Z}(\widehat{\Z/S\Z}, \mathbb{Z}) = \textrm{Hom}(\widehat{ \Z/S\Z }, \Q / \mathbb{Z}) / \textrm{Im}(0) \cong \textrm{Hom}(\widehat{Z/S\Z}, \Q / \mathbb{Z}) \cong \widehat{\widehat{\Z / S \Z}}.  
\end{equation*}
Since the double Pontryagin dual is isomorphic to the original group, we obtain the result.  
\end{proof}

Note that the above calculation holds if $\widehat{\Z / S \Z}$ is replaced by any pure torsion group. 

\begin{example}
$\textrm{Ext}_\mathbb{Z}^1( G_S, \mathbb{Z}) \cong (\Z / S \Z) / \Z$.  Here $G_S : =\left \{ k/l \in \Q : k\in\Z, l | S\right\}$ is the group from the previous section.
\end{example}
\begin{proof}
Notice that we have the following short exact sequence: 

\begin{equation*}
    \begin{tikzcd}
0 \arrow{r} & \mathbb{Z} \arrow{r} & G_S  \arrow{r}{\phi} & \widehat{\Z/S\Z} \arrow{r} & 0
\end{tikzcd}
\end{equation*}
where $\phi(k/l) = e^{2 \pi i k / l}$.  Hence, by 2) of Proposition \ref{homalg}, we can obtain the exact sequence
\begin{equation*}
    0 \to \textrm{Hom}(G_S,\mathbb{Z}) \to \Z \to \textrm{Ext}^1_\mathbb{Z}(\widehat{\Z / S \Z},\mathbb{Z}) \to \textrm{Ext}^1_\mathbb{Z}(G_S,\mathbb{Z}) \to 0
\end{equation*}
where we used that $\textrm{Ext}^1_\mathbb{Z}(\mathbb{Z}, \Z) \cong 0$ since $\Z$ is free. We have also seen in Lemma \ref{HomLemma} that $\textrm{Hom}(G_S, \Z) \cong 0$, so the above sequence is short exact. Therefore, by the previous example:
$$\textrm{Ext}^1_\mathbb{Z}(G_S, \mathbb{Z}) \cong \textrm{Ext}^1_\mathbb{Z}(\widehat{\Z / S \Z},\mathbb{Z})/ \Z \cong  (\Z / S \Z) / \Z.$$
\end{proof}

\end{document}